\documentclass[a4paper,10pt]{amsart}
\usepackage{pb-diagram}
\usepackage{xypic}
\usepackage{amsfonts,amscd,amssymb,amsmath,latexsym,amscd,bbm,amsthm,mathrsfs,hyperref}
\usepackage{amscd}
\usepackage{bm}
\usepackage{amsmath}
\usepackage{latexsym}
\usepackage{amsfonts}
\usepackage{amsthm}
\usepackage{indentfirst}
\usepackage{algorithm}
\usepackage{algorithmic}
\usepackage{ragged2e}

 \newtheorem{thm}{Theorem}[section]

 \newtheorem{prop}[thm]{Proposition}

 \def\k{\mathbbm{k}}

 \newcommand{\Hom}{\mathrm{Hom}}

\title[Derived equivalences of DG algebras and their cohomology rings]{Derived equivalences of DG algebras are not governed by their cohomology rings}
\author{X.-F. Mao}
\address{Department of Mathematics, Shanghai University, Shanghai 200444, China}
\email{xuefengmao@shu.edu.cn}

\author{R.-K. Lu}
\address{Department of Mathematics, Shanghai University, Shanghai 200444, China}
\email{lrk.wer@gmail.com}

\date{}

\begin{document}

\maketitle \def\abstactname{abstact}
\begin{abstract}
Let $\mathscr{A}$ and $\mathscr{B}$ be two connected cochain DG algebra such that $\mathscr{A}^{\#}=\mathscr{B}^{\#}$ and the cohomology rings $H(\mathscr{A})$ and  $H(\mathscr{B})$ are isomorphic.  We give examples to show that $\mathscr{A}$ and $\mathscr{B}$ are not necessarily derived equivalent, thereby answering to a question of Dugas.
\end{abstract}

\maketitle
\section*{introduction}
Let $\mathscr{A}$ and $\mathscr{B}$ be cochain DG algebras.
In \cite[Section $7$]{Dug}, Dugas proposed the problem: if the cohomology rings $H(\mathscr{A})$ and $H(\mathscr{B})$ are derived equivalent, then whether the DG algebras $\mathscr{A}$ and $\mathscr{B}$ are derived equivalent or not. Motivated by the problem above, Pan, Peng and Zhang \cite{PPZ} construct derived equivalences of differential graded algebras which are endomorphism algebras of the objects from a triangle in the homotopy category of DG algebras. And the cohomology rings of these differential graded endomorphism algebras are also derived equivalent under some conditions. This gives an affirmative answer to the above problem in some special cases. However, it is still an open question whether a derived equivalence of the cohomology rings implies the DG algebras are derived equivalent. One can see the problem was mentioned again in \cite{PPZ}.

The motivation of this paper is to show that $\mathscr{A}$ and $\mathscr{B}$ are not necessarily derived equivalent even under the hypothesis that $H(\mathscr{A})\cong H(\mathscr{B})$ and $\mathscr{A}^{\#}=\mathscr{B}^{\#}$. To be more specific, we find concrete counter-examples in DG free algebras,  which were introduced and systematically studied in \cite{MXYA}. Recall that a connected cochain DG algebra $\mathscr{A}$ is called DG free if its underlying graded algebra $$\mathscr{A}^{\#}=\k\langle x_1,x_2,\cdots, x_n\rangle,\,\, \text{with}\,\, |x_i|=1,\,\, \forall i\in \{1,2,\cdots, n\}.$$ By \cite[Theorem 2.4]{MXYA},  $\partial_{\mathscr{A}}$ is uniquely determined by a crisscrossed ordered $n$-tuple of $n\times n$ matrixes.
Recall that an ordered $n$-tuple $(M^1,M^2,\cdots, M^n)$ of $n\times n$ matrixes with each
\begin{align*}
M^i=(c^i_1,c^i_2,\cdots,c^i_n)=\left(
                         \begin{array}{c}
                           r^i_1 \\
                           r^i_2  \\
                           \vdots  \\
                           r^i_n
                         \end{array}
                       \right)
\end{align*}
is called crisscross if
$
\sum\limits_{k=1}^n[c_j^kr_k^i-c_k^ir_j^k]=(0)_{n\times n}, \forall i,j\in \{1,2,\cdots,n\}.
$
To be precise and specific, the differential $\partial_{\mathscr{A}}$ of a DG free algebra with $n$ generators associated with a crisscross ordered $n$-tuple $(M^1,M^2,\cdots, M^n)$ is
defined  by
$$\partial_{\mathscr{A}}(x_i)=(x_1,x_2,\cdots, x_n)M^i\left(
                         \begin{array}{c}
                           x_1 \\
                           x_2  \\
                           \vdots  \\
                           x_n
                         \end{array}
                       \right), i=1,2,\cdots,n. $$
\\
\begin{bfseries}
Example\,
\end{bfseries}
Let $\mathscr{A}_1$ and $\mathscr{A}_2$ be two DG free algebras such that $\mathscr{A}_1^{\#}=\k\langle x_1,x_2,x_3\rangle$, $\partial_{\mathscr{A}_1}(x_1)=x_3^2,\partial_{\mathscr{A}_1}(x_2)=x_2^2,\partial_{\mathscr{A}_1}(x_3)=0$ and
$\mathscr{A}_2^{\#}=\k\langle y_1,y_2,y_3\rangle$, $\partial_{\mathscr{A}_2}(y_1)=y_3^2$, $\partial_{\mathscr{A}_2}(y_2)=y_1y_3+y_3y_1$, $\partial_{\mathscr{A}_2}(y_3)=0$. The crisscross ordered $3$-tuples associated with $\mathscr{A}_1$ and $\mathscr{A}_2$ are respectively
$$\left(\left(
    \begin{array}{ccc}
      0 & 0 & 0 \\
      0 & 0 & 0 \\
      0 & 0 & 1 \\
    \end{array}
  \right), \left(
    \begin{array}{ccc}
      0 & 0 & 0 \\
      0 & 1 & 0 \\
      0 & 0 & 0 \\
    \end{array}
  \right), \left(
    \begin{array}{ccc}
      0 & 0 & 0 \\
      0 & 0 & 0 \\
      0 & 0 & 0 \\
    \end{array}
  \right)\right)$$
and
$$\left(\left(
    \begin{array}{ccc}
      0 & 0 & 0 \\
      0 & 0 & 0 \\
      0 & 0 & 1 \\
    \end{array}
  \right), \left(
    \begin{array}{ccc}
      0 & 0 & 1 \\
      0 & 0 & 0 \\
      1 & 0 & 0 \\
    \end{array}
  \right), \left(
    \begin{array}{ccc}
      0 & 0 & 0 \\
      0 & 0 & 0 \\
      0 & 0 & 0 \\
    \end{array}
  \right)\right).$$
Applying \cite[Theorem $3.1$]{MXYA}, one sees that $\mathscr{A}_1$ and $\mathscr{A}_2$ are not isomorphic.
In Section \ref{Aone} and Section \ref{Atwo}, we prove
 $$H(\mathscr{A}_1)=\frac{\k[\lceil x_3 \rceil, \lceil x_1x_3+x_3x_1 \rceil]}{( \lceil x_3\rceil^2 )}\,\, \text{and}\,\, H(\mathscr{A}_2)=\frac{\k[\lceil y_3\rceil,\lceil y_1^2+y_2y_3+y_3y_2\rceil]}{(\lceil y_3\rceil^2 )}.$$
To explain that $\mathscr{A}_1$ and $\mathscr{A}_2$ are not derived equivalent, we compute the derived Picard group of $\mathscr{A}_1$ and $\mathscr{A}_2$ by applying the calculation methods from \cite{MYH}. We have  $$\mathrm{DPic}(\mathscr{A}_1)\cong \mathbb{Z}\times  \left\{\left(
                                            \begin{array}{cc}
                                               a & b \\
                                               0 & a^2 \\
                                            \end{array}
                                          \right)\quad| \quad a\in \k^{\times},b\in \k\right\}$$
 and $$\mathrm{DPic}(\mathscr{A}_2)\cong \mathbb{Z}\times \left\{\left(
                                                                                       \begin{array}{ccc}
                                                                                         a & b & c \\
                                                                                         0 & a^2 & 2ab \\
                                                                                         0 & 0 & a^3 \\
                                                                                       \end{array}
                                                                                   \right)|\,\, a\in \k^*,b, c\in \k\right\}.$$
One sees that $\mathrm{DPic}(\mathscr{A}_1)\not \cong \mathrm{DPic}(\mathscr{A}_2)$. Since derived Picard group of an algebra is a derived invariant, we conclude that $\mathscr{A}_1$ and $\mathscr{A}_2$ are not derived equivalent. This gives a negative answer to Dugas's problem. The reason for this phenomenon lies in the fact that much intrinsic information of a non-formal DG algebra $\mathscr{A}$ is missing after taking the cohomology functor. To retrieve the lost information, we should endow the cohomology ring with suitable $A_{\infty}$-algebra structures. Indeed, it is well known that $H(\mathscr{A})$ admits an $A_{\infty}$-algebra structure called the minimal model of $\mathscr{A}$ such that there is an $\mathscr{A}_{\infty}$-quasi-isomorphism $\mathscr{A}\to H(\mathscr{A})$ (cf. \cite{Kad, Mer, Kel2}).

\section{some basics on dg modules and dg algebras}
In this section, we introduce some notation and conventions on DG algebras and DG modules.
There is some
overlap here with the papers \cite{MW1, MW2,FJ}. It is assumed that
the reader is familiar with basics on DG modules, triangulated
categories and derived categories. If this is not the case, we refer
to  \cite{Nee, Wei} for more details on them.

Throughout this paper, $\k$ is an algebraically closed field of characteristic $0$.
For any $\k$-vector space $V$, we write $V^*=\Hom_{\k}(V,\k)$. Assume that $\{e_i|i\in I\}$ is a basis of a finite-dimensional $\k$-vector space $V$.  We denote the dual basis of $V$ by $\{e_i^*|i\in I\}$, i.e., $\{e_i^*|i\in I\}$ is a basis of $V^*$ such that $e_i^*(e_j)=\delta_{i,j}$. For any graded vector space $W$ and $j\in\Bbb{Z}$, the $j$-th suspension $\Sigma^j W$ of $W$ is a graded vector space defined by $(\Sigma^j W)^i=W^{i+j}$.

A cochain DG algebra is
a graded
$\k$-algebra $\mathscr{A}$ together with a differential $\partial_{\mathscr{A}}: \mathscr{A}\to \mathscr{A}$  of
degree $1$ such that
\begin{align*}
\partial_{\mathscr{A}}(ab) = (\partial_{\mathscr{A}} a)b + (-1)^{|a|}a(\partial_{\mathscr{A}} b)
\end{align*}
for all graded elements $a, b\in \mathscr{A}$.
For any DG algebra $\mathscr{A}$,  we denote $\mathscr{A}\!^{op}$ as its opposite DG
algebra, whose multiplication is defined as
 $a \cdot b = (-1)^{|a|\cdot|b|}ba$ for all
graded elements $a$ and $b$ in $\mathscr{A}$. We denote by $\mathscr{A}^i$ its $i$-th homogeneous component.  The differential $\partial_{\mathscr{A}}$ is a sequence of linear maps $\partial_{\mathscr{A}}^i: \mathscr{A}^i\to \mathscr{A}^{i+1}$ such that $\partial_{\mathscr{A}}^{i+1}\circ \partial_{\mathscr{A}}^i=0$, for all $i\in \Bbb{Z}$.  If $\partial_{\mathscr{A}}\neq 0$, $\mathscr{A}$ is called
non-trivial. The cohomology graded algebra of $\mathscr{A}$ is the graded algebra $$H(\mathscr{A})=\bigoplus_{i\in \Bbb{Z}}\frac{\mathrm{ker}(\partial_{\mathscr{A}}^i)}{\mathrm{im}(\partial_{\mathscr{A}}^{i-1})}.$$
 For any cocycle element $z\in \mathrm{ker}(\partial_{\mathscr{A}}^i)$, we write $\lceil z \rceil$ as the cohomology class in $H(\mathscr{A})$ represented by $z$.  A cochain algebra $\mathscr{A}$ is called connected if its underlying graded algebra $\mathscr{A}^{\#}$ is a connected graded algebra. One sees that $H(\mathscr{A})$ is a connected graded algebra if $\mathscr{A}$ is a connected cochain DG algebra.
 For any connected DG algebra $\mathscr{A}$,
   we write $\frak{m}$ as the maximal DG ideal $\mathscr{A}^{>0}$ of $\mathscr{A}$.
Via the canonical surjection $\varepsilon: \mathscr{A}\to \k$, $\k$ is both a DG
$\mathscr{A}$-module and a DG $\mathscr{A}\!^{op}$-module. It is easy to check that the enveloping DG algebra $\mathscr{A}^e = \mathscr{A}\otimes \mathscr{A}\!^{op}$ of $\mathscr{A}$
is also a connected cochain DG algebra with $H(\mathscr{A}^e)\cong H(\mathscr{A})^e$,  $$\partial_{\mathscr{A}^e}=\partial_{\mathscr{A}}\otimes \mathscr{A}^{op} + \mathscr{A}\otimes
\partial_{\mathscr{A}^{op}} \quad \text{and}\quad \frak{m}_{\mathscr{A}^e}=\frak{m}_{\mathscr{A}}\otimes \mathscr{A}^{op}+\mathscr{A}\otimes \frak{m}_{\mathscr{A}^{op}}.$$

  The derived category of left DG modules over $\mathscr{A}$ (DG $\mathscr{A}$-modules for short) is denoted by $\mathscr{D}(\mathscr{A})$.  A DG $\mathscr{A}$-module  $M$ is compact if the functor $\Hom_{\mathscr{D}(\mathscr{A})}(M,-)$ preserves
all coproducts in $\mathscr{D}(\mathscr{A})$.
 By \cite[Proposition 3.3]{MW2},
a DG $\mathscr{A}$-module  is compact if and only if it admits a minimal semi-free resolution with a finite semi-basis. The full subcategory of $\mathscr{D}(\mathscr{A})$ consisting of compact DG $\mathscr{A}$-modules is denoted by $\mathscr{D}^c(\mathscr{A})$.

 Let $\mathscr{A}$ be a connected cochain DG algebra.  A DG $\mathscr{A}^e$-module $X$ is called tilting if there is a DG $\mathscr{A}^e$-module $Y$ such that $X\otimes_{\mathscr{A}}^L Y\cong \mathscr{A}$ and $Y\otimes_{\mathscr{A}}^LX\cong \mathscr{A}$ in $\mathscr{D}(\mathscr{A}^e)$. The DG $\mathscr{A}^e$-module $Y$ is called a quasi-inverse of $X$. It is easy to see that the quasi-inverse of a tilting DG module is also tilting. And the quasi-inverse of a given tilting DG $\mathscr{A}^e$-module is unique up to isomorphism in $\mathscr{D}(\mathscr{A}^e)$.
If $X_1$ and $X_2$ are two tilting DG modules, then so is $X_1\otimes_{\mathscr{A}}^L X_2$ by the associativity of $-\otimes_{\mathscr{A}}^L-$.
A DG $\mathscr{A}^e$-module $X$ is tilting iff any one of the following conditions holds (cf. \cite[Proposition $2.3$]{MYH}).
\begin{enumerate}
\item The functors $X\otimes_{\mathscr{A}}^L-$ and $-\otimes_{\mathscr{A}}^L X$ are auto-equivalences of $\mathscr{D}(\mathscr{A}^{e})$.
\item The functors $X\otimes_{\mathscr{A}}^L-$ and $-\otimes_{\mathscr{A}}^L X$ are auto-equivalences of $\mathscr{D}(\mathscr{A})$ and $\mathscr{D}(\mathscr{A}^{op})$ respectively.
\item The functors $X\otimes_{\mathscr{A}}^L-$ and $-\otimes_{\mathscr{A}}^L X$ are auto-equivalences of $\mathscr{D}^c(\mathscr{A})$ and $\mathscr{D}^c(\mathscr{A}^{op})$ respectively.
\item $\langle {}_{\mathscr{A}}X\rangle =\mathscr{D}^c(\mathscr{A})$,  $\langle X_A\rangle = \mathscr{D}^c(\mathrm{DGmod}\,\,A^{op})$,
 and the adjunction morphisms $$\mathscr{A}\to R\Hom_{\mathscr{A}}(X,X) \quad \text{and}  \quad \mathscr{A}\to R\Hom_{\mathscr{A}^{op}}(X,X)$$ in $\mathscr{D}^c(\mathscr{A}^{e})$ are isomorphisms.
\end{enumerate}
The derived Picard group $\mathrm{DPic}(\mathscr{A})$ of $\mathscr{A}$ is defined as the abelian group,  whose elements are the isomorphism classes of tilting DG $\mathscr{A}^e$-modules in $\mathscr{D}(\mathscr{A}^{e})$. In $\mathrm{DPic}(\mathscr{A})$, the product of the classes of $X$ and $Y$ is given by the class of $X\otimes_{\mathscr{A}}^L Y$. Note that the unit element of $\mathrm{DPic}(\mathscr{A})$ is the class of $\mathscr{A}$ and $\mathrm{DPic}(\mathscr{A})$ is a derived invariant. The derived Picard group of an algebra is introduced independently by Yekutieli and Rouquier-Zimmermann \cite{Yek1,RZ}.
Due to the well known Rickard's theory \cite{Ric1,Ric2}, one sees that the derived Picard group of an algebra is an invariant of derived category.
There has been many works on properties and computations of derived Picard groups of various kinds of algebras. For examples, the derived Picard group of a commutative unital ring has been computed in some cases \cite{Kel3,RZ,Yek1,Fau}, some properties of the derived Picard groups of an order are given in \cite{Zim}, the derived Picard group of a finite dimensional algebra over an algebraic closed field is proved to be a locally algebraic group in \cite{Yek2},  and some structures and calculations of the derived Picard groups of finite dimensional hereditary algebras over an algebraic closed field are presented in \cite{MY}. In \cite{Kel1}, Keller interprets Hochschild cohomology as the Lie algebra of the derived Picard group and
deduces that it is preserved under derived equivalences. Recently, Volkov and Zvonareva \cite{VZ} compute the derived Picard groups of selfinjective Nakayama algebras.  It is also interesting for one to consider the case of DG algebras.
For graded commutative DG algebras, Yekutieli \cite{Yek3} has done some research on their derived Picard groups and dualizing DG modules.
In \cite{MYH}, it is proved that the derived Picard group of a homologically smooth and Koszul connected cochain DG algebra is isomorphic to the opposite group of the derived category of its Ext-algebra.

In the rest of this section, we review some definitions of various homological properties of DG algebras. Let $\mathscr{A}$ be a connected cochain DG algebra.
\begin{enumerate}
\item  If $\dim_{k}H(R\Hom_{\mathscr{A}}(\k,\mathscr{A}))=1$ (resp. $\dim_{k}H(R\Hom_{\mathscr{A}\!^{op}}(\k,\mathcal{A}))=1)$, then $\mathscr{A}$ is called left (resp. right) Gorenstein (cf.\cite{MW2});
\item  If ${}_{\mathscr{A}}k$, or equivalently ${}_{\mathscr{A}^e}\mathscr{A}$, has a minimal semi-free resolution with a semi-basis concentrated in degree $0$, then $\mathscr{A}$ is called Koszul (cf. \cite{HW});
\item If ${}_{\mathscr{A}}k$, or equivalently the DG $\mathscr{A}^e$-module $\mathscr{A}$ is compact, then $\mathscr{A}$ is called homologically smooth (cf. \cite[Corollary 2.7]{MW3});
\item If $\mathscr{A}$ is homologically smooth and $$R\Hom_{\mathscr{A}^e}(\mathcal{A}, \mathscr{A}^e)\cong
\Sigma^{-n}\mathscr{A}$$ in  the derived category $\mathrm{D}((\mathscr{A}^e)^{op})$ of right DG $\mathscr{A}^e$-modules, then $\mathscr{A}$ is called an $n$-Calabi-Yau DG algebra  (cf. \cite{Gin,VdB}).
\end{enumerate}
Note that $\mathscr{A}$ is called Gorenstein
if $\mathscr{A}$ is both left Gorenstein and right Gorenstein.  We emphasize that homologically smooth DG algebras are called `regular' DG algebras in \cite{HW} and \cite{MW2}.  For any homologically smooth DG algebra $\mathscr{A}$, it is left Gorenstein if and only if it is right Gorenstein by \cite[Remark 7.6]{MW2}.
Let $\mathscr{A}$ be a Koszul cochain DG algebra. He-Wu proved in \cite{HW} that $\mathscr{A}$ is homologically smooth and Gorenstein if and only if its $\mathrm{Ext}$-algebra $H(R\Hom_{\mathscr{A}}(\mathbbm{k},\mathbbm{k}))$ is a Frobenius algebra.
Recently, the first author \cite{Mao}  shows that the result is also right for non-Koszul case.
A special family of Gorenstein homologically smooth DG algebras are called Calabi-Yau.
 Since Ginzburg introduced Calabi-Yau (DG) algebras in \cite{Gin}, they have been extensively studied due to their links to
mathematical physics, representation theory and non-commutative algebraic geometry.
By \cite[Proposition 6.4]{MXYA}, any Calabi-Yau DG algebra is Gorenstein. While the converse is generally not true. The first author and J.-W. He \cite{HM} showed that a Koszul connected cochain DG algebra to is Calabi-Yau if and only if its Ext-algebra is a symmetric Frobenius algebra. It has been generalized to non-Koszul cases in \cite{Mao}.

\section{$H(\mathscr{A}_1)$ and $\mathrm{DPic}(\mathscr{A}_1)$}\label{Aone}
In this section, we will compute $H(\mathscr{A}_1)$, $\mathrm{DPic}(\mathscr{A}_1)$ and study the homological properties of $\mathscr{A}_1$.
\begin{prop}\label{cohaone}The cohomology ring $H(\mathscr{A}_1)$  of $\mathscr{A}_1$ is
$$\frac{\k[\lceil x_3 \rceil, \lceil x_1x_3+x_3x_1 \rceil]}{( \lceil x_3\rceil^2 )}.$$
\end{prop}
\begin{proof}
The differential of $\mathscr{A}_1$ is defined by $\partial_{\mathscr{A}_1}(x_1)=x_3^2, \partial_{\mathscr{A}_1}(x_2)=x_2^2,\partial_{\mathscr{A}_1}(x_3)=0$.   It is easy for one to see that $H^0(\mathscr{A}_1)=\k$ and $H^1(\mathscr{A}_1)=\k\lceil x_3\rceil$. Since $$\mathscr{A}_1^2=\bigoplus\limits_{i=1}^3\bigoplus\limits_{j=1}^3\k x_ix_j,$$ any cocycle element in $\mathscr{A}_1^2$ can be written as $\sum\limits_{i=1}^3\sum\limits_{j=1}^3c_{ij}x_ix_j$ for some $c_{ij}\in \k$ and $i,j\in \{1,2,3\}$. We have
\begin{align*}
0&=\partial_{\mathscr{A}_1}(\sum\limits_{i=1}^3\sum\limits_{j=1}^3c_{ij}x_ix_j)=c_{11}x_3^2x_1-c_{11}x_1x_3^2+c_{12}x_3^2x_2-c_{12}x_1x_2^2\\
& +(c_{13}-c_{31})x_3^3+c_{21}x_2^2x_1-c_{21}x_2x_3^2+c_{23}x_2^2x_3-c_{32}x_3x_2^2.
\end{align*}
We have $c_{11}=c_{12}=c_{21}=c_{23}=c_{32}=0$ and $c_{13}=c_{31}$. Hence
$$\mathrm{ker}(\partial_{\mathscr{A}_1}^2)= \k(x_1x_3+x_3x_1)\oplus \k x_2^2\oplus \k x_3^2.$$ Since $\mathrm{im}(\partial_{\mathscr{A}_1}^1)=\k x_2^2\oplus \k x_3^2$, we have $H^2(\mathscr{A}_1)=\k \lceil x_1x_3+x_3x_1\rceil$.
Assume that we have proved $H^{2i-1}(\mathscr{A}_1)=\k \lceil x_3(x_1x_x+x_3x_1)^{i-1}\rceil$ and $H^{2i}(\mathscr{A}_1)=\k \lceil (x_1x_3+x_3x_1)^i\rceil$, $i\ge 1$. We should show that \begin{align*}
H^{2i+1}(\mathscr{A}_1)&=\k \lceil x_3(x_1x_3+x_3x_1)^i\rceil  \\
 \text{and}\,\, H^{2i+2}(\mathscr{A}_1)&=\k \lceil (x_1x_3+x_3x_1)^{i+1}\rceil.
\end{align*}
 For any cocycle element in $\mathscr{A}_1^{2i+1}$, we may write it as $x_1a_1+x_2a_2+x_3a_3$, $a_1,a_2,a_3\in \mathscr{A}_1^{2i}$.
 We have \begin{align*}
 0=&\partial_{\mathscr{A}_1}(x_1a_1+x_2a_2+x_3a_3)\\
 =&x_3^2a_1-x_1\partial_{\mathscr{A}_1}(a_1)+x_2^2a_2-x_2\partial_{\mathscr{A}_1}(a_2)-x_3\partial_{\mathscr{A}_1}(a_3) \\
 =&x_3[x_3a_1-\partial_{\mathscr{A}_1}(a_3)]+x_2[x_2a_2-\partial_{\mathscr{A}_1}(a_2)]-x_1\partial_{\mathscr{A}_1}(a_1).
 \end{align*}
 Then $\partial_{\mathscr{A}_1}(a_3)=x_3a_1$, $\partial_{\mathscr{A}_1}(a_2)=x_2a_2$ and $\partial_{\mathscr{A}_1}(a_1)=0$. Since $Z^{2i}(\mathscr{A}_1)\cong H^{2i}(\mathscr{A}_1)\oplus B^{2i}(\mathscr{A}_1)$, we may let $a_1=l(x_1x_3+x_3x_1)^i+\partial_{\mathscr{A}_1}(\chi)$ for some $l\in \k$ and $\chi\in \mathscr{A}_1^{2i-1}$.
 We have $\partial_{\mathscr{A}_1}(a_3)=lx_3(x_1x_3+x_3x_1)^i+x_3\partial_{\mathscr{A}_1}(\chi)$.  Then $a_3=x_3\lambda +\omega$ for some $\lambda\in \mathscr{A}^{2i}$, and $\omega\in Z^{2i}(\mathscr{A}_1)$. We get $$\partial_{\mathscr{A}_1}(\chi)+l(x_1x_3+x_3x_1)^i=\partial_{\mathscr{A}_1}(\lambda).$$
 If $l\neq 0$, we get $\lceil x_1x_3+x_3x_1\rceil^i=0$ in $H^{2i}(\mathscr{A}_1)$ which contradicts with $H^{2i}(\mathscr{A}_1)=\k \lceil (x_1x_3+x_3x_1)^i\rceil$. So $l=0$ and $a_3=-x_3\chi+t(x_1x_3+x_3x_1)^i+\partial_{\mathscr{A}_1}(\eta)$, for some $t\in \k$ and $\eta\in \mathscr{A}_1^{2i-1}$. Since $\partial_{\mathscr{A}_1}(a_2)=x_2a_2$, we have $a_2=x_2\delta$ for some $\delta\in Z^{2i-1}(\mathscr{A})$.

  Then
 \begin{align*}
 \lceil x_1a_1+x_2a_2 +x_3a_3\rceil & =\lceil x_1 \partial_{\mathscr{A}_1}(\chi)+x_2^2\delta-x_3^2\chi +tx_3(x_1x_3+x_3x_1)^i+x_3\partial_{\mathscr{A}_1}(\eta)\rceil \\
 &=\lceil  \partial_{\mathscr{A}_1}(-x_1\chi+x_2\delta-x_1\eta)  +tx_3(x_1x_3+x_3x_1)^i \rceil \\
 &=t\lceil x_3(x_1x_3+x_3x_1)^i   \rceil.
 \end{align*}
 So $H^{2i+1}(\mathscr{A}_1)=\k\lceil x_3(x_1x_3+x_3x_1)^i   \rceil$.

 Now, let $x_1b_1+x_2b_2+x_3b_3$ be an arbitrary cocycle element in $\mathscr{A}_1^{2i+2}$. Then we have
   \begin{align*}
 0=&\partial_{\mathscr{A}_1}(x_1b_1+x_2b_2+x_3b_3)\\
 =&x_3^2b_1-x_1\partial_{\mathscr{A}_1}(b_1)+x_2^2b_2-x_2\partial_{\mathscr{A}_1}(b_2)-x_3\partial_{\mathscr{A}_1}(b_3) \\
 =&x_3[x_3b_1-\partial_{\mathscr{A}_1}(b_3)]-x_1\partial_{\mathscr{A}_1}(b_1)+x_2[x_2b_2-\partial_{\mathscr{A}_1}(b_2)].
 \end{align*}
 So $\partial_{\mathscr{A}_1}(b_3)=x_3b_1$, $\partial_{\mathscr{A}_1}(b_1)=0$ and $x_2b_2=\partial_{\mathscr{A}_1}(b_2)$. Since
 $$Z^{2i+1}(\mathscr{A}_1)\cong H^{2i+1}(\mathscr{A}_1)\oplus B^{2i+1}(\mathscr{A}_1),$$ we may let $b_1=rx_3(x_1x_3+x_3x_1)^i+\partial_{\mathscr{A}_1}(\omega)$ for some $r\in \k$ and $\omega\in \mathscr{A}_1^{2i}$.
 Then $\partial_{\mathscr{A}_1}(b_3)=rx_3^2(x_1x_3+x_3x_1)^i+x_3\partial_{\mathscr{A}_1}(\omega)$.
 So $$b_3=rx_1(x_1x_3+x_3x_1)^i-x_3\omega+sx_3(x_1x_3+x_3x_1)^i+\partial_{\mathscr{A}_1}(\varphi)$$ for some $s\in \k$ and $\varphi\in \mathscr{A}_1^{2i}$. Since $x_2b_2=\partial_{\mathscr{A}_1}(b_2)$, we have $b_2=x_2\beta$ for some $\beta\in Z^{2i+1}(\mathscr{A})$.
 Therefore,
 \begin{align*}
 &\quad \quad \lceil x_1b_1+x_2b_2 +x_3b_3\rceil \\
 &=\lceil r(x_1x_3+x_3x_1)^{i+1}+\partial_{\mathscr{A}_1}[-x_1\omega+x_2\beta+sx_1(x_1x_3+x_3x_1)^i-x_3\varphi]\rceil\\
 &= r\lceil (x_1x_3+x_3x_1)^{i+1}\rceil.
 \end{align*}
 Thus $H^{2i+2}(\mathscr{A}_1)=\k\lceil (x_1x_3+x_3x_1)^{i+1}   \rceil$. By the induction above, we obtain that
 \begin{align*}
H^{2n-1}(\mathscr{A}_1)&=\k \lceil x_3(x_1x_3+x_3x_1)^{n-1}\rceil  \\
 \text{and}\,\, H^{2n}(\mathscr{A}_1)&=\k \lceil (x_1x_3+x_3x_1)^{n}\rceil, \forall n\ge 1.
\end{align*}
 Since $x_3(x_1x_3+x_3x_1)-(x_1x_3+x_3x_1)x_3=x_3^2x_1-x_1x_3^2=\partial_{\mathscr{A}_1}(x_1^2)$, we have
 $$\lceil x_3\rceil \cdot \lceil x_1x_3+x_3x_1\rceil = \lceil x_1x_3+x_3x_1\rceil\cdot \lceil x_3\rceil$$ in $H(\mathscr{A}_1)$. Hence $H(\mathscr{A}_1)=\k[\lceil x_3\rceil,\lceil x_1x_3+x_3x_1\rceil]/(\lceil x_3\rceil^2)$.
\end{proof}

By Proposition \ref{cohaone}, one sees that $H(\mathscr{A}_1)$ is an AS-Gorenstein graded algebra.
Hence $\mathscr{A}_1$ is a Gorenstein DG algebra by \cite[Proposition 1]{Gam}. It is natural for one to ask whether it is Koszul, homologically smooth and Calabi-Yau. From the proof of the following proposition, the readers will see that we rely heavily on the construction of the minimal semi-free resolution of ${}_{\mathscr{A}_1}\k$.

\begin{prop}\label{aone}
The DG free algebra $\mathscr{A}_1$ is a Kozul Calabi-Yau DG algebra with
 $$\mathrm{DPic}(\mathscr{A}_1)\cong \mathbb{Z}\times  \left\{\left(
                                            \begin{array}{cc}
                                               a & b \\
                                               0 & a^2 \\
                                            \end{array}
                                          \right)\quad| \quad a\in \k^{\times},b\in \k\right\}.$$
\end{prop}
\begin{proof}
By definition,  $\mathscr{A}_1^{\#}=\k\langle x_1,x_2,x_3\rangle$ and the differential $\partial_{\mathscr{A}_1}$ is defined by $$\partial_{\mathscr{A}_1}(x_1)=x_3^2, \partial_{\mathscr{A}_1}(x_2)=x_2^2, \partial_{\mathscr{A}_1}(x_3)=0.$$
According to the constructing procedure of the minimal semi-free resolution in \cite[Proposition 2.4]{MW1}, we get a minimal semi-free resolution $f: F\stackrel{\simeq}{\to} {}_{\mathscr{A}_1}\k$, where $F$ is a semi-free DG $\mathscr{A}_1$-module with $$F^{\#}=\mathscr{A}_1^{\#}\oplus \mathscr{A}_1^{\#}\Sigma e_{x_3}\oplus \mathscr{A}_1^{\#}\Sigma e_z, \quad \partial_{F}(\Sigma e_{x_3})=x_3, \quad \partial_{F}(\Sigma e_{z})=x_1+x_3\Sigma e_{x_3};$$ and $f$ is defined by $f|_{\mathscr{A}_1}=\varepsilon$, $f(\Sigma e_{x_3})=0$ and $f(\Sigma e_z)=0$.

We should prove that $f$ is a quasi-isomorphism. It suffices to show $H(F)=\k$.
For any graded cocycle element $a_z\Sigma e_z + a_{x_3}\Sigma e_{x_3}+a\in Z^{2k}(F)$, we have
\begin{align*}
0&=\partial_F(a_z\Sigma e_z + a_{x_3}\Sigma e_{x_3}+a)\\
&=\partial_{\mathscr{A}_1}(a_z)\Sigma e_z +a_z(x_1+x_3\Sigma e_{x_3})+\partial_{\mathscr{A}_1}(a_{x_3})\Sigma e_{x_3}+a_{x_3}x_3+\partial_{\mathscr{A}_1}(a)\\
&=\partial_{\mathscr{A}_1}(a_z)\Sigma e_z +[a_zx_3+\partial_{\mathscr{A}_1}(a_{x_3})]\Sigma e_{x_3}+a_zx_1+a_{x_3}x_3+\partial_{\mathscr{A}_1}(a).
\end{align*}
This implies that
$$
\begin{cases}
\partial_{\mathscr{A}_1}(a_z)=0 \\
a_zx_3+\partial_{\mathscr{A}_1}(a_{x_3})=0 \\
a_zx_1+a_{x_3}x_3+\partial_{\mathscr{A}_1}(a)=0.
\end{cases}
$$ Since $$Z^{2k}(\mathscr{A}_1)=H^{2k}(\mathscr{A}_1)\oplus B^{2k}(\mathscr{A}_1)$$ and $$H(\mathscr{A}_1)=\k[\lceil x_3\rceil, \lceil x_1x_3+x_3x_1\rceil ]/(\lceil x_3\rceil^2),$$
we have $a_z=\partial_{\mathscr{A}_1}(c)+t(x_1x_3+x_3x_1)^k$ for some $c\in \mathscr{A}_1^{2k-1}, t\in \k$. Then
$$[\partial_{\mathscr{A}_1}(c)+t(x_1x_3+x_3x_1)^k]x_3+\partial_{\mathscr{A}_1}(a_{x_3})=0. $$ Hence $t=0$ and $\partial_{\mathscr{A}_1}(c)x_3+\partial_{\mathscr{A}_1}(a_{x_3})=0$. Then $$a_{x_3}=-cx_3+\partial_{\mathscr{A}_1}(\mu)+s(x_1x_3+x_3x_1)^k$$
for some $\mu\in \mathscr{A}_1^{2k-1}$ and $s\in \k$.
We have
$$\partial_{\mathscr{A}_1}(c)x_1+[-cx_3+\partial_{\mathscr{A}_1}(\mu)+s(x_1x_3+x_3x_1)^k]x_3+\partial_{\mathscr{A}_1}(a)=0,$$
which implies that $\partial_{\mathscr{A}_1}(c)=0, s=0$,
and $\partial_{\mathscr{A}_1}(a)=cx_3^2-\partial_{\mathscr{A}_1}(\mu)x_3$. Then $a_z=0$, $a_{x_3}=-cx_3+\partial_{\mathscr{A}_1}(\mu)$ and $$a=-cx_1-\mu x_3+\partial_{\mathscr{A}_1}(\lambda)+\tau(x_1x_3+x_3x_1)^k$$
 for some $\lambda\in \mathscr{A}_1^{2k-1}, \tau\in \k$. Since $c\in Z^{2k-1}(\mathscr{A}_1)\cong H^{2k-1}(\mathscr{A}_1)\oplus B^{2k-1}(\mathscr{A}_1)$,
 we may let
 $c=\partial_{\mathscr{A}_1}(\chi)+\omega(x_1x_3+x_3x_1)^{k-1}x_3$, for some $\chi\in \mathscr{A}_1^{2k-2}$ and $\omega \in \k$.
Therefore,
\begin{align*}
& a_z\Sigma e_z + a_{x_3}\Sigma e_{x_3}+a\\
=&[-cx_3+\partial_{\mathscr{A}_1}(\mu)]\Sigma e_{x_3}-cx_1-\mu x_3+\partial_{\mathscr{A}_1}(\lambda)+\tau(x_1x_3+x_3x_1)^k\\
=&\{-[\partial_{\mathscr{A}_1}(\chi)+\omega(x_1x_3+x_3x_1)^{k-1}x_3]x_3+\partial_{\mathscr{A}_1}(\mu)\}\Sigma e_{x_3} \\
 & -[\partial_{\mathscr{A}_1}(\chi)+\omega(x_1x_3+x_3x_1)^{k-1}x_3]x_1-\mu x_3+\partial_{\mathscr{A}_1}(\lambda)+\tau(x_1x_3+x_3x_1)^k               \\
 =&\partial_F[(\omega-\tau)(x_1x_3+x_3x_1)^{k-1}x_3\Sigma e_z ]\\
 &+ \partial_F\{[-\chi x_3+\mu-\tau(x_1x_3+x_3x_1)^{k-1}x_1]\Sigma e_{x_3}-\chi x_1+\lambda\}.
\end{align*}
Hence $H^{2k}(F)=0$, for any $k\in \Bbb{N}$. It remains to show
 $H^{2k-1}(F)=0$, for any $k\in \Bbb{N}$.  Let $a_z\Sigma e_z + a_{x_3}\Sigma e_{x_3}+a\in Z^{2k-1}(F)$, we have
\begin{align*}
0&=\partial_F(a_z\Sigma e_z + a_{x_3}\Sigma e_{x_3}+a)\\
&=\partial_{\mathscr{A}_1}(a_z)\Sigma e_z -a_z(x_1+x_3\Sigma e_{x_3})+\partial_{\mathscr{A}_1}(a_{x_3})\Sigma e_{x_3}-a_{x_3}x_3+\partial_{\mathscr{A}_1}(a)\\
&=\partial_{\mathscr{A}_1}(a_z)\Sigma e_z +[-a_zx_3+\partial_{\mathscr{A}_1}(a_{x_3})]\Sigma e_{x_3}-a_zx_1-a_{x_3}x_3+\partial_{\mathscr{A}_1}(a).
\end{align*}
This implies that
$$
\begin{cases}
\partial_{\mathscr{A}_1}(a_z)=0 \\
-a_zx_3+\partial_{\mathscr{A}_1}(a_{x_3})=0 \\
-a_zx_1-a_{x_3}x_3+\partial_{\mathscr{A}_1}(a)=0.
\end{cases}
$$ Since $$Z^{2k-1}(\mathscr{A}_1)=H^{2k-1}(\mathscr{A}_1)\oplus B^{2k-1}(\mathscr{A}_1)$$ and $$H(\mathscr{A}_1)=\k[\lceil x_3\rceil, \lceil x_1x_3+x_3x_1\rceil ]/(\lceil x_3\rceil^2),$$
we have $a_z=\partial_{\mathscr{A}_1}(c)+t(x_1x_3+x_3x_1)^{k-1}x_3$ for some $c\in \mathscr{A}_1^{2k-2}, t\in \k$. Then
$$-[\partial_{\mathscr{A}_1}(c)+t(x_1x_3+x_3x_1)^{k-1}x_3]x_3+\partial_{\mathscr{A}_1}(a_{x_3})=0. $$ Then
$$a_{x_3}=cx_3+t(x_1x_3+x_3x_1)^{k-1}x_1+\partial_{\mathscr{A}_1}(\mu)+s(x_1x_3+x_3x_1)^{k-1}x_3$$
for some $\mu\in \mathscr{A}_1^{2k-2}$ and $s\in \k$.
We have
\begin{align*}
0=&-[cx_3+t(x_1x_3+x_3x_1)^{k-1}x_1+\partial_{\mathscr{A}_1}(\mu)+s(x_1x_3+x_3x_1)^{k-1}x_3]x_3\\
&-[\partial_{\mathscr{A}_1}(c)+t(x_1x_3+x_3x_1)^{k-1}x_3]x_1+\partial_{\mathscr{A}_1}(a),
\end{align*}
which implies that $t=0$, $a_z=\partial_{\mathscr{A}_1}(c)$, $a_{x_3}=cx_3+\partial_{\mathscr{A}_1}(\mu)+s(x_1x_3+x_3x_1)^{k-1}x_3$ and $a=cx_1+\mu x_3+s(x_1x_3+x_3x_1)^{k-1}x_1+\partial_{\mathscr{A}_1}(\lambda)+\tau(x_1x_3+x_3x_1)^{k-1}x_3$,
 for some $\lambda\in \mathscr{A}_1^{2k-2}, \tau\in \k$. Hence
\begin{align*}
& a_z\Sigma e_z + a_{x_3}\Sigma e_{x_3}+a\\
=&\partial_{\mathscr{A}_1}(c)\Sigma e_{z} +[cx_3+\partial_{\mathscr{A}_1}(\mu)+s(x_1x_3+x_3x_1)^{k-1}x_3] \Sigma e_{x_3} \\
&+cx_1+\mu x_3+s(x_1x_3+x_3x_1)^{k-1}x_1+\partial_{\mathscr{A}_1}(\lambda)+\tau(x_1x_3+x_3x_1)^{k-1}x_3 \\
 =&\partial_F\{[c+s(x_1x_3+x_3x_1)^{k-1}]\Sigma e_z +[\mu+\tau(x_1x_3+x_3x_1)^{k-1}] \Sigma e_{x_3}+\lambda\}.
\end{align*}
Hence $H^{2k-1}(F)=0$, for any $k\in \Bbb{N}$. Therefore, $f$ is a quasi-isomorphism.

Since $F$ has a semi-basis $\{1, \Sigma e_{x_3},\Sigma e_z\}$ concentrated in degree $0$,  $\mathscr{A}_1$ is a Koszul homologically smooth DG algebra.
By the minimality of $F$, we have $$H(\Hom_{\mathscr{A}_1}(F,\k))=\Hom_{\mathscr{A}_1}(F,\k)= \k\cdot 1^*\oplus \k\cdot(\Sigma e_{x_3})^*\oplus \k \cdot(\Sigma e_z)^*.$$  So the Ext-algebra $E=H(\Hom_{\mathscr{A}_1}(F,F))$  is concentrated in degree $0$.
On the other hand, $$\Hom_{\mathscr{A}_1}(F,F)^{\#}\cong (\k \cdot 1^*\oplus \k \cdot (\Sigma e_{x_3})^*\oplus \k \cdot (\Sigma e_z)^*)\otimes_{\k} F^{\#}$$ is concentrated in degree $\ge 0$. This implies that $E= Z^0(\Hom_{\mathscr{A}_1}(F,F))$.
Since $F^{\#}$ is a free graded $\mathscr{A}_1^{\#}$-module with a basis $\{1,\Sigma e_{x_3},\Sigma e_z\}$ concentrated in degree $0$,
  the elements in  $\Hom_{\mathscr{A}_1}(F,F)^0$ are in one to one correspondence with the matrices in $M_3(\k)$. Indeed, any $f\in \Hom_{\mathscr{A}_1}(F_{\k},F_{\k})^0$ is uniquely determined by
  a matrix $A_f=(a_{ij})_{3\times 3}\in M_3(\k)$ with
$$\left(
                         \begin{array}{c}
                          f(1) \\
                          f(\Sigma e_{x_3})\\
                          f(\Sigma e_z)\\
                         \end{array}
                       \right) =      A_f \cdot \left(
                         \begin{array}{c}
                          1 \\
                          \Sigma e_{x_3}\\
                          \Sigma e_z\\
                         \end{array}
                       \right).  $$
                       We see $f\in  Z^0(\Hom_{\mathscr{A}_1}(F,F)$ if and only if $\partial_{F}\circ f=f\circ \partial_{F}$, if and only if
 $$ A_f\cdot \left(
                         \begin{array}{ccc}
                           0 & 0& 0 \\
                           x_3 & 0 & 0 \\
                           x_1 & x_3 & 0 \\
                         \end{array}
                       \right) =  \left(
                         \begin{array}{cccc}
                           0 & 0& 0 \\
                           x_3 & 0 & 0 \\
                           x_1 & x_3 & 0 \\
                         \end{array}
                       \right) \cdot A_f, $$  which is also equivalent to
                       $$\begin{cases}
                       a_{12}=a_{13}=a_{23}=0\\
                       a_{11}=a_{22}=a_{33}\\
                       a_{21}=a_{32}
                       \end{cases}$$
by a direct computation. Let $$  \mathcal{E}=\left\{ \left(
                         \begin{array}{ccc}
                           a & 0& 0\\
                           b & a & 0 \\
                           c & b & a \\
                         \end{array}
                       \right)\quad | \quad a,b,c,\in \k \right\}$$ be the subalgebra of the matrix algebra. Then one sees $E\cong \mathcal{E}$.
                       Set \begin{align*} e_1= \left(
                         \begin{array}{ccc}
                           1 & 0& 0\\
                           0 & 1 & 0 \\
                           0 & 0 & 1 \\
                         \end{array}
                       \right),& e_2= \left(
                         \begin{array}{ccc}
                           0 & 0& 0\\
                           1 & 0 & 0 \\
                           0 & 1 & 0 \\
                         \end{array}
                       \right),
                        e_3= \left(
                         \begin{array}{ccc}
                           0 & 0& 0\\
                           0 & 0 & 0 \\
                           1 & 0 & 0 \\
                         \end{array}
                       \right).
                       \end{align*}
                     Then $\{e_1,e_2,e_3\}$ is a $\k$-linear bases of the $\k$-algebra
                        $\mathcal{E}$. The multiplication on $\mathcal{E}$ is defined by the following relations
                       $$\begin{cases} e_1\cdot e_i=e_i\cdot e_1=e_i, i=1,2,3 \\
                        e_2^2=e_3, e_2\cdot e_3=e_3\cdot e_2=0
                       \end{cases} .$$
                       So $\mathcal{E}$ is a local commutative $\k$-algebra isomorphic to $\k[X]/(X^3)$.  Hence
$E\cong \k[X]/(X^3)$ is a symmetric Frobenius algebra concentrated
in degree $0$. This implies that
$\mathrm{Tor}_{\mathscr{A}_1}^0(\k_{\mathscr{A}_1},{}_{\mathscr{A}_1}\k)\cong
E^*$ is a symmetric coalgebra. By \cite[Theorem 4.2]{HM},
$\mathscr{A}_1$ is a Koszul Calabi-Yau DG algebra. Since $\{e_1,e_2,e_3\}$ is a $\k$-linear basis of $\mathcal{E}$, any $\k$-linear map $\sigma: \mathcal{E}\to \mathcal{E}$ uniquely corresponds to a matrix in $A_{\sigma}=(a_{ij})_{3\times 3}\in M_3(\k)$ with
$$ \left(
     \begin{array}{c}
       \sigma(e_1) \\
       \sigma(e_2) \\
       \sigma(e_3) \\
     \end{array}
   \right)=A_{\sigma}\left(
                       \begin{array}{c}
                         e_1 \\
                         e_2 \\
                         e_3 \\
                       \end{array}
                     \right).$$ Such $\sigma\in \mathrm{Aut}_{\k}(\mathcal{E})$ if and only if
                     $$A_{\sigma}\in \mathrm{GL}_3(\k) \quad \text{and}\quad \sigma(e_i\cdot e_j)=\sigma(e_i)\sigma(e_j), \, \text{for any}\, \, i,j =1,2,3.$$
Therefore, $\sigma\in \mathrm{Aut}_{\k}(\mathcal{E})$ if and only if
$$ \begin{cases}|(a_{ij})_{3\times 3}|\neq 0, \sigma(e_1)=e_1\\
[\sigma(e_2)]^2=\sigma(e_3),[\sigma(e_3)]^2=0\\
\sigma(e_2)\cdot \sigma(e_3)=\sigma(e_3)\cdot \sigma(e_2)=0
\end{cases} \quad \Leftrightarrow \quad \begin{cases} c_{22}\neq 0, c_{11}=1,c_{12}=c_{13}=0\\
c_{21}=c_{31}=c_{32}=0\\
c_{33}=c_{22}^2.
\end{cases}$$
Hence $$\mathrm{Aut}_{\k}\mathcal{E}\cong \left\{\left(
                                            \begin{array}{ccc}
                                              1 & 0 & 0 \\
                                              0 & a & b \\
                                              0 & 0 & a^2 \\
                                            \end{array}
                                          \right)\quad| \quad a\in \k^{\times},b\in \k\right\}.$$
Since $\mathcal{E}$ is commutative, we have $\mathrm{Aut}_{\k}(\mathcal{E})\cong \mathrm{Out}_{\k}(\mathcal{E})$. By \cite[Remark 4.4]{MYH}, we have
$\mathrm{Pic}_{\k}(\mathcal{E})\cong \mathrm{Out}_{\k}(\mathcal{E})$ and $\mathrm{DPic}_{\k}(\mathcal{E})=\mathbb{Z}\times \mathrm{Pic}_{\k}(\mathcal{E})$. Thus \begin{align*}
\mathrm{DPic}(\mathscr{A})\cong \mathrm{DPic}_{\k}(E) &\cong \mathrm{DPic}_{\k}(\mathcal{E})\cong \mathbb{Z}\times \left\{\left(
                                            \begin{array}{ccc}
                                              1 & 0 & 0 \\
                                              0 & a & b \\
                                              0 & 0 & a^2 \\
                                            \end{array}
                                          \right)\quad| \quad a\in \k^{\times},b\in \k\right\} \\
&\cong \mathbb{Z}\times \left\{\left(
                                            \begin{array}{cc}
                                               a & b \\
                                               0 & a^2 \\
                                            \end{array}
                                          \right)\quad| \quad a\in \k^{\times},b\in \k\right\}
\end{align*}
by \cite[Theorem 4.3]{MYH}.

\end{proof}

\section{$H(\mathscr{A}_2)$ and $\mathrm{DPic}(\mathscr{A}_2)$ }\label{Atwo}
In this section, we will compute $H(\mathscr{A}_2)$, $\mathrm{DPic}(\mathscr{A}_2)$ and study the homological properties of $\mathscr{A}_2$.
The reader will see the arrangement of this section is similar to Section \ref{Aone}. Comparatively speaking, the computations for $\mathscr{A}_2$ is a little more complicated.
\begin{prop}\label{cohatwo}The cohomology ring $H(\mathscr{A}_2)$ of $\mathscr{A}_2$ is
$$H(\mathscr{A}_2)=\frac{\k[\lceil y_3\rceil,\lceil y_1^2+y_2y_3+y_3y_2\rceil]}{(\lceil y_3\rceil^2 )}.$$
\end{prop}

\begin{proof}
Since $\partial_{\mathscr{A}_2}$ is defined by $\partial_{\mathscr{A}_2}(y_1)=y_3^2, \partial_{\mathscr{A}_2}(y_2)=y_1y_3+y_3y_1,\partial_{\mathscr{A}_2}(y_3)=0$,  one sees easily that $H^0(\mathscr{A}_2)=\k$ and $H^1(\mathscr{A}_2)=\k\lceil y_3\rceil$. Note that $$\mathscr{A}_2^2=\bigoplus\limits_{i=1}^3\bigoplus\limits_{j=1}^3\k y_iy_j.$$  Each cocycle element in $\mathscr{A}_2^2$ can be written as $\sum\limits_{i=1}^3\sum\limits_{j=1}^3c_{ij}y_iy_j$ for some $c_{ij}\in \k$ and  $i,j\in \{1,2,3\}$. We have
\begin{align*}
0&=\partial_{\mathscr{A}_2}(\sum\limits_{i=1}^3\sum\limits_{j=1}^3c_{ij}y_iy_j)=c_{11}y_3^2y_1-c_{11}y_1y_3^2+c_{12}y_3^2y_2-c_{12}y_1(y_1y_3+y_3y_1)\\
&\quad\quad +(c_{13}-c_{31})y_3^3+c_{21}(y_1y_3+y_3y_1)y_1-c_{21}y_2y_3^2 + c_{22}(y_1y_3+y_3y_1)y_2\\
&\quad\quad -c_{22}y_2(y_1y_3+y_3y_1)+c_{23}(y_1y_3+y_3y_1)y_3-c_{32}y_3(y_1y_3+y_3y_1)\\
&=(c_{11}-c_{32})y_3^2y_1+(c_{23}-c_{11})y_1y_3^2+c_{12}y_3^2y_2-c_{12}y_1^2y_3+(c_{21}-c_{12})y_1y_3y_1\\
&\quad\quad +(c_{13}-c_{31})y_3^2+(c_{23}-c_{32})y_3y_1y_3 +c_{21}y_3y_1^2-c_{21}y_2y_3^2+c_{22}y_1y_3y_2+c_{22}y_3y_1y_2\\
&\quad\quad -c_{22}y_2y_1y_3-c_{22}y_2y_3y_1
\end{align*}
We have $c_{11}=c_{32}=c_{23}$,  $c_{12}=c_{21}=c_{22}=0$ and $c_{13}=c_{31}$. Hence
$$\mathrm{ker}(\partial_{\mathscr{A}_1}^2)= \k(y_1^2+y_2y_3+y_3y_2)\oplus \k(y_1y_3+y_3y_1)\oplus \k y_3^2.$$ Since $\mathrm{im}(\partial_{\mathscr{A}_1}^1)=\k (y_1y_3+y_3y_1)\oplus \k y_3^2$, we have $H^2(\mathscr{A}_1)=\k \lceil y_1^2+ y_2y_3+y_3y_2\rceil$.
Assume that we have proved $H^{2i-1}(\mathscr{A}_2)=\k \lceil y_3(y_1^2+y_2y_3+y_3y_2)^{i-1}\rceil$ and $H^{2i}(\mathscr{A}_2)=\k \lceil (y_1^2+y_2y_3+y_3y_2)^i\rceil$, $i\ge 1$. We should show that \begin{align*}
H^{2i+1}(\mathscr{A}_2)&=\k \lceil y_3(y_1^2+y_2y_3+y_3y_2)^i\rceil  \\
 \text{and}\,\, H^{2i+2}(\mathscr{A}_2)&=\k \lceil (y_1^2+y_2y_3+y_3y_2)^{i+1}\rceil.
\end{align*}
 For any cocycle element in $\mathscr{A}_2^{2i+1}$, we may write it as $y_1a_1+y_2a_2+y_3a_3$, $a_1,a_2,a_3\in \mathscr{A}_2^{2i}$.
 We have \begin{align*}
 0=&\partial_{\mathscr{A}_2}(y_1a_1+y_2a_2+y_3a_3)\\
 =&y_3^2a_1-y_1\partial_{\mathscr{A}_2}(a_1)+(y_1y_3+y_3y_1)a_2-y_2\partial_{\mathscr{A}_2}(a_2)-y_3\partial_{\mathscr{A}_2}(a_3) \\
 =&y_3[y_3a_1+y_1a_2-\partial_{\mathscr{A}_2}(a_3)]-y_2\partial_{\mathscr{A}_2}(a_2)+y_1[y_3a_2-\partial_{\mathscr{A}_2}(a_1)].
 \end{align*}
 Then $\partial_{\mathscr{A}_2}(a_3)=y_3a_1+y_1a_2$, $\partial_{\mathscr{A}_2}(a_2)=0$ and $\partial_{\mathscr{A}_2}(a_1)=y_3a_2$. Since $Z^{2i}(\mathscr{A}_1)\cong H^{2i}(\mathscr{A}_1)\oplus B^{2i}(\mathscr{A}_1)$, we may let $a_2=l(y_1^2+y_2y_3+y_3y_2)^i+\partial_{\mathscr{A}_2}(\chi)$ for some $l\in \k$ and $\chi\in \mathscr{A}_2^{2i-1}$.
 We have $\partial_{\mathscr{A}_2}(a_1)=ly_3(y_1^2+y_2y_3+y_3y_2)^i+y_3\partial_{\mathscr{A}_2}(\chi)$.  Then $l=0$ and $a_1=-y_3\chi+\theta(y_1^2+y_2y_3+y_3y_2)^i+\partial_{\mathscr{A}_2}(\omega)$ for some $\theta\in \k$ and $\omega\in \mathscr{A}_2^{2i-1}$.
 Since $$\partial_{\mathscr{A}_2}(a_3)=y_3a_1+y_1a_2=y_3[-y_3\chi+\theta(y_1^2+y_2y_3+y_3y_2)^i+\partial_{\mathscr{A}_2}(\omega)]+y_1\partial_{\mathscr{A}_2}(\chi),$$ we have $\theta=0$ and $a_3=-y_1\chi -y_3\omega + t(y_1^2+y_2y_3+y_3y_2)^i+\partial_{\mathscr{A}_2}(\eta)$ for some $t\in \k$ and $\eta\in \mathscr{A}_2^{2i-1}$. Then
 \begin{align*}
  \lceil y_1a_1+y_2a_2 +y_3a_3\rceil & =\lceil \partial_{\mathscr{A}_2}(-y_2\chi-y_1\omega-y_3\eta)+ty_3(y_1^2+y_2y_3+y_3y_2)^i\rceil \\
 &=\lceil ty_3(y_1^2+y_2y_3+y_3y_2)^i \rceil.
 \end{align*}
 So $H^{2i+1}(\mathscr{A}_2)=\k\lceil y_3(y_1^2+y_2y_3+y_3y_2)^i \rceil$.

 Now, let $y_1b_1+y_2b_2+y_3b_3$ be an arbitrary cocycle element in $\mathscr{A}_2^{2i+2}$. Then we have
   \begin{align*}
 0=&\partial_{\mathscr{A}_2}(y_1b_1+y_2b_2+y_3b_3)\\
 =&y_3^2b_1-y_1\partial_{\mathscr{A}_1}(b_1)+(y_1y_3+y_3y_1)b_2-y_2\partial_{\mathscr{A}_2}(b_2)-y_3\partial_{\mathscr{A}_2}(b_3) \\
 =&y_3[y_3b_1+y_1b_2-\partial_{\mathscr{A}_2}(b_3)]+y_1[y_3b_2-\partial_{\mathscr{A}_1}(b_1)]-y_2\partial_{\mathscr{A}_2}(b_2).
 \end{align*}
 So $\partial_{\mathscr{A}_2}(b_2)=0$, $\partial_{\mathscr{A}_2}(b_1)=y_3b_2$ and $y_3b_1+y_1b_2=\partial_{\mathscr{A}_2}(b_3)$. Since
 $$Z^{2i+1}(\mathscr{A}_2)\cong H^{2i+1}(\mathscr{A}_2)\oplus B^{2i+1}(\mathscr{A}_2),$$ we may let
 $b_2=ry_3(y_1^2+y_2y_3+y_3y_2)^i+\partial_{\mathscr{A}_2}(\omega)$ for some $r\in \k$ and $\omega\in \mathscr{A}_2^{2i}$.
 Then $$\partial_{\mathscr{A}_2}(b_1)=y_3b_2=ry_3^2(y_1^2+y_2y_3+y_3y_2)^i+y_3\partial_{\mathscr{A}_2}(\omega).$$
 So $b_1=ry_1(y_1^2+y_2y_3+y_3y_2)^i-y_3\omega+sy_3(y_1^2+y_2y_3+y_3y_2)^i+\partial_{\mathscr{A}_2}(\varphi)$ for some $s\in \k$ and $\varphi\in \mathscr{A}_2^{2i}$.
  Since \begin{align*}
  \partial_{\mathscr{A}_2}(b_3)&=y_3b_1+y_1b_2 \\
  &=y_3[ry_1(y_1^2+y_2y_3+y_3y_2)^i-y_3\omega+sy_3(y_1^2+y_2y_3+y_3y_2)^i+\partial_{\mathscr{A}_2}(\varphi)]\\
  &\quad \quad + y_1[ry_3(y_1^2+y_2y_3+y_3y_2)^i+\partial_{\mathscr{A}_2}(\omega)]\\
  &=\partial_{\mathscr{A}_2}[ry_2(y_1^2+y_2y_3+y_3y_2)^i-y_1\omega+sy_1(y_1^2+y_2y_3+y_3y_2)^i-y_3\varphi],
  \end{align*}
 we have \begin{align*}
 b_3=ry_2(y_1^2+y_2y_3+y_3y_2)^i-y_1\omega+sy_1(y_1^2+y_2y_3+y_3y_2)^i-y_3\varphi \\
 +\lambda y_3(y_1^2+y_2y_3+y_3y_2)^i+\partial_{\mathscr{A}_2}(\xi)
 \end{align*} for some $\xi \in \mathscr{A}_2^{2i}$ and $\lambda\in \k$.
 Therefore,
 \begin{align*}
 &\quad\quad y_1b_1+y_2b_2 +y_3b_3\\
 &=y_1[ry_1(y_1^2+y_2y_3+y_3y_2)^i-y_3\omega+sy_3(y_1^2+y_2y_3+y_3y_2)^i+\partial_{\mathscr{A}_2}(\varphi)]\\
 &\quad + y_2[ry_3(y_1^2+y_2y_3+y_3y_2)^i+\partial_{\mathscr{A}_2}(\omega)]+\lambda y_3^2(y_1^2+y_2y_3+y_3y_2)^i+y_3\partial_{\mathscr{A}_2}(\xi)\\
 &\quad + y_3[ry_2(y_1^2+y_2y_3+y_3y_2)^i-y_1\omega+sy_1(y_1^2+y_2y_3+y_3y_2)^i-y_3\varphi] \\
 &=r(y_1^2+y_2y_3+y_3y_2)^{i+1}+\partial_{\mathscr{A}_2}[sy_2(y_1^2+y_2y_3+y_3y_2)^i-y_2\omega-y_1\varphi-y_3\xi]\\
 &\quad + \partial_{\mathscr{A}_2}[\lambda y_1(y_1^2+y_2y_3+y_3y_2)^i].
 \end{align*}
 Thus $H^{2i+2}(\mathscr{A}_2)=\k\lceil (y_1^2+y_2y_3+y_3y_2)^{i+1}   \rceil$. By the induction above, we obtain that
 \begin{align*}
H^{2n-1}(\mathscr{A}_2)&=\k \lceil y_3(y_1^2+y_2y_3+y_3y_2)^{n-1}\rceil  \\
 \text{and}\,\, H^{2n}(\mathscr{A}_1)&=\k \lceil (y_1^2+y_2y_3+y_3y_2)^{n}\rceil, \forall n\ge 1.
\end{align*}
 Since \begin{align*}
 y_3(y_1^2+y_2y_3+y_3y_2)-(y_1^2+y_2y_3+y_3y_2)y_3&=y_3y_1^2-y_1^2y_3+y_3^2y_2-y_2y_3^2\\
 &=\partial_{\mathscr{A}_1}(y_1y_2+y_2y_1),
 \end{align*} we have
 $\lceil y_3\rceil \cdot \lceil y_1^2+y_2y_3+y_3y_2\rceil = \lceil y_1^2+y_2y_3+y_3y_2\rceil\cdot \lceil y_3\rceil$ in $H(\mathscr{A}_2)$. Hence $$H(\mathscr{A}_2)=\frac{\k[\lceil y_3\rceil,\lceil y_1^2+y_2y_3+y_3y_2\rceil]}{(\lceil y_3\rceil^2)}.$$

\end{proof}

By Proposition \ref{cohatwo}, one sees that $H(\mathscr{A}_2)$ is an AS-Gorenstein graded algebra.
Hence $\mathscr{A}_2$ is a Gorenstein DG algebra by \cite[Proposition 1]{Gam}. Furthermore, we have the following proposition on its other homological properties and derived Picard group.
\begin{prop}\label{atwo}
The DG free algebra $\mathscr{A}_2$ is a Kozul Calabi-Yau DG algebra with
$$\mathrm{DPic}(\mathscr{A}_2)\cong \mathbb{Z}\times \left\{\left(
                                                                                       \begin{array}{ccc}
                                                                                         a & b & c \\
                                                                                         0 & a^2 & 2ab \\
                                                                                         0 & 0 & a^3 \\
                                                                                       \end{array}
                                                                                   \right)|\,\, a\in \k^*,b, c\in \k\right\}.$$
\end{prop}

\begin{proof}
By definition,  $\mathscr{A}_2^{\#}=\k\langle y_1,y_2,y_3\rangle$ and the differential $\partial_{\mathscr{A}_2}$ is defined by $$\partial_{\mathscr{A}_2}(y_1)=y_3^2, \partial_{\mathscr{A}_2}(y_2)=y_1y_3+y_3y_1, \partial_{\mathscr{A}_2}(y_3)=0.$$
According to the constructing procedure of the minimal semi-free resolution in \cite[Proposition 2.4]{MW1}, we get a minimal semi-free resolution $f: F\stackrel{\simeq}{\to} {}_{\mathscr{A}_2}\k$, where $F$ is a semi-free DG $\mathscr{A}_2$-module with $$F^{\#}=\mathscr{A}_2^{\#}\oplus \mathscr{A}_2^{\#}\Sigma e_{y_3}\oplus \mathscr{A}_2^{\#}\Sigma e_z\oplus \mathscr{A}_2^{\#}\Sigma e_{r}, \,\, \partial_{F}(\Sigma e_{y_3})=y_3,$$
$$ \partial_{F}(\Sigma e_{z})=y_1+y_3\Sigma e_{y_3},\quad \partial_{F}(\Sigma e_r)=y_3\Sigma e_z+y_1\Sigma e_{y_3}+y_2,$$ and $f$ is defined by $f|_{\mathscr{A}_1}=\varepsilon$, $f(\Sigma e_{y_3})=0$, $f(\Sigma e_z)=0$ and $f(\Sigma e_r)=0$. We should prove that $f$ is a quasi-isomorphism. It suffices to show $H(F)=\k$.
For any graded cocycle element $a_r\Sigma e_r + a_z\Sigma e_z + a_{y_3}\Sigma e_{y_3} + a\in Z^{2k}(F)$, we have
\begin{align*}
0&=\partial_F(a_r\Sigma e_r + a_z\Sigma e_z + a_{y_3}\Sigma e_{y_3}+a)\\
&=\partial_{\mathscr{A}_2}(a_r)\Sigma e_r+a_r[y_3\Sigma e_z+y_1\Sigma e_{y_3}+y_2]+\partial_{\mathscr{A}_2}(a_z)\Sigma e_z +a_z(y_1+y_3\Sigma e_{y_3})\\
&\quad +\partial_{\mathscr{A}_2}(a_{y_3})\Sigma e_{y_3}+a_{y_3}y_3+\partial_{\mathscr{A}_2}(a)\\
&=\partial_{\mathscr{A}_2}(a_r)\Sigma e_r+[\partial_{\mathscr{A}_2}(a_z)+a_ry_3]\Sigma e_z +[a_ry_1+a_zy_3+\partial_{\mathscr{A}_2}(a_{y_3})]\Sigma e_{y_3}\\
&\quad +a_ry_2+a_zy_1+a_{y_3}y_3+\partial_{\mathscr{A}_2}(a).
\end{align*}
This implies that
$$
\begin{cases}
\partial_{\mathscr{A}_2}(a_r)=0 \\
a_ry_3+\partial_{\mathscr{A}_2}(a_{z})=0 \\
a_ry_1+a_{z}y_3+\partial_{\mathscr{A}_2}(a_{y_3})=0\\
a_ry_2+a_zy_1+a_{y_3}y_3+\partial_{\mathscr{A}_2}(a)=0.
\end{cases}
$$
Since $Z^{2k}(\mathscr{A}_2)=H^{2k}(\mathscr{A}_2)\oplus B^{2k}(\mathscr{A}_2)$ and $H(\mathscr{A}_2)=\frac{\k[\lceil y_3\rceil, \lceil y_1^2+y_2y_3+y_3y_2\rceil ]}{(\lceil y_3\rceil^2)},$
we have $a_r=\partial_{\mathscr{A}_2}(c)+t(y_1^2+y_2y_3+y_3y_2)^k$ for some $c\in \mathscr{A}_2^{2k-1}, t\in \k$. Then
$$[\partial_{\mathscr{A}_2}(c)+t(y_1^2+y_2y_3+y_3y_2)^k]y_3+\partial_{\mathscr{A}_2}(a_{z})=0. $$ Hence $t=0$ and $a_{z}=-cy_3+\partial_{\mathscr{A}_2}(\mu)+s(y_1^2+y_2y_3+y_3y_2)^k$
for some $\mu\in \mathscr{A}_2^{2k-1}$ and $s\in \k$.
We have
$$\partial_{\mathscr{A}_2}(c)y_1+[-cy_3+\partial_{\mathscr{A}_2}(\mu)+s(y_1^2+y_2y_3+y_3y_2)^k]y_3+\partial_{\mathscr{A}_2}(a_{y_3})=0,$$
which implies that $s=0$
and $\partial_{\mathscr{A}_2}(a_{y_3})+\partial_{\mathscr{A}_2}(cy_1+\mu y_3)=0$. Then  $$a_{y_3}=-cy_1-\mu y_3+\partial_{\mathscr{A}_2}(\lambda)+\tau(y_1^2+y_2y_3+y_3y_2)^k$$
 for some $\lambda\in \mathscr{A}_2^{2k-1}, \tau\in \k$.
 Since \begin{align*}
&\quad\quad  \partial_{\mathscr{A}_2}(a)=-a_ry_2-a_zy_1-a_{y_3}y_3\\
 &=-\partial_{\mathscr{A}_2}(c)y_2-[\partial_{\mathscr{A}_2}(\mu)-cy_3]y_1-[\tau(y_1^2+y_2y_3+y_3y_2)^k-cy_1-\mu y_3+\partial_{\mathscr{A}_2}(\lambda)]y_3\\
 &=\partial_{\mathscr{A}_2}[-cy_2-\mu y_1-\lambda y_3]-\tau(y_1^2+y_2y_3+y_3y_2)^ky_3,
 \end{align*}
we conclude that $\tau=0$ and $a=-cy_2-\mu y_1-\lambda y_3+\theta(y_1^2+y_2y_3+y_3y_2)^k+\partial_{\mathscr{A}_2}(\xi)$ for some $\theta\in \k$ and $\xi \in \mathscr{A}_2^{2k-1}$.
Therefore,
\begin{align*}
&a_r\Sigma e_r + a_z\Sigma e_z + a_{y_3}\Sigma e_{y_3}+a\\
=&\partial_{\mathscr{A}_2}(c)\Sigma e_r +[\partial_{\mathscr{A}_2}(\mu)-cy_3]\Sigma e_z+[\partial_{\mathscr{A}_2}(\lambda)-cy_1-\mu y_3]\Sigma e_{y_3}\\
&\quad +\partial_{\mathscr{A}_2}(\xi)-cy_2-\mu y_1-\lambda y_3+\theta (y_1^2+y_2y_3+y_3y_2)^k\\
=&\partial_F[c\Sigma e_r+\mu \Sigma e_z +\lambda \Sigma e_{y_3}+\xi-\theta(y_1^2+y_2y_3+y_3y_2)^{k-1}(y_3\Sigma e_r+y_1\Sigma e_z+y_2\Sigma e_{y_3}) ].
\end{align*}
Hence $H^{2k}(F)=0$, for any $k\in \Bbb{N}$. It remains to show
 $H^{2k-1}(F)=0$, for any $k\in \Bbb{N}$.  Let $b_r\Sigma e_r + b_z\Sigma e_z + b_{y_3}\Sigma e_{y_3}+b\in Z^{2k-1}(F)$, we have
\begin{align*}
0&=\partial_F(b_r\Sigma e_r+b_z\Sigma e_z + b_{x_3}\Sigma e_{x_3}+b)\\
&=\partial_{\mathscr{A}_2}(b_r)\Sigma e_r + [\partial_{\mathscr{A}_2}(b_z)-b_ry_3]\Sigma e_z +[\partial_{\mathscr{A}_2}(b_{y_3})-b_ry_1-b_zy_3]\Sigma e_{y_3}\\
&\quad +\partial_{\mathscr{A}_2}(b)-b_ry_2-b_zy_1-b_{y_3}y_3.
\end{align*}
This implies that
$$
\begin{cases}
\partial_{\mathscr{A}_2}(b_r)=0 \\
\partial_{\mathscr{A}_2}(b_{z})=b_ry_3\\
\partial_{\mathscr{A}_2}(b_{y_3})=b_ry_1+b_zy_3\\
\partial_{\mathscr{A}_2}(b)=b_ry_2+b_zy_1+b_{y_3}y_3.
\end{cases}
$$ Since $Z^{2k-1}(\mathscr{A}_2)=H^{2k-1}(\mathscr{A}_2)\oplus B^{2k-1}(\mathscr{A}_2)$ and $$H(\mathscr{A}_2)=\frac{\k[\lceil y_3\rceil, \lceil y_1^2+y_2y_3+y_3y_2\rceil ]}{(\lceil y_3\rceil^2)},$$
we have $b_r=\partial_{\mathscr{A}_2}(c)+t(y_1^2+y_2y_3+y_3y_2)^{k-1}y_3$ for some $c\in \mathscr{A}_2^{2k-2}, t\in \k$. Then
$$\partial_{\mathscr{A}_2}(b_{z})=[\partial_{\mathscr{A}_2}(c)+t(y_1^2+y_2y_3+y_3y_2)^{k-1}y_3]y_3.$$
Hence $b_z=cy_3+t(y_1^2+y_2y_3+y_3y_2)^{k-1}y_1 +s(y_1^2+y_2y_3+y_3y_2)^{k-1}y_3+\partial_{\mathscr{A}_2}(\mu)$
for some $\mu\in \mathscr{A}_2^{2k-2}$ and $s\in \k$.
Then \begin{align*}
 & \quad\quad\partial_{\mathscr{A}_2}(b_{y_3})=b_ry_1+b_zy_3=[\partial_{\mathscr{A}_2}(c)+t(y_1^2+y_2y_3+y_3y_2)^{k-1}y_3]y_1\\
& +[cy_3+t(y_1^2+y_2y_3+y_3y_2)^{k-1}y_1 +s(y_1^2+y_2y_3+y_3y_2)^{k-1}y_3+\partial_{\mathscr{A}_2}(\mu)]y_3\\
&=\partial_{\mathscr{A}_2}[cy_1+t(y_1^2+y_2y_3+y_3y_2)^{k-1}y_2+s(y_1^2+y_2y_3+y_3y_2)^{k-1}y_1+\mu y_3].
\end{align*}
So $b_{y_3}=cy_1+(y_1^2+y_2y_3+y_3y_2)^{k-1}(ty_2+sy_1)+\mu y_3+\lambda (y_1^2+y_2y_3+y_3y_2)^{k-1}y_3+\partial_{\mathscr{A}_2}(\xi)$ for some $\lambda\in \k$ and $\xi\in \mathscr{A}_2^{2k-2}$. Since
\begin{align*}
&\quad \quad \partial_{\mathscr{A}_2}(b)=b_ry_2+b_zy_1+b_{y_3}y_3=[\partial_{\mathscr{A}_2}(c)+t(y_1^2+y_2y_3+y_3y_2)^{k-1}y_3]y_2\\
&+[cy_3+t(y_1^2+y_2y_3+y_3y_2)^{k-1}y_1 +s(y_1^2+y_2y_3+y_3y_2)^{k-1}y_3+\partial_{\mathscr{A}_2}(\mu)]y_1\\
&+[cy_1+(y_1^2+y_2y_3+y_3y_2)^{k-1}(ty_2+sy_1+\lambda y_3)+\mu y_3+\partial_{\mathscr{A}_2}(\xi)]y_3\\
&= \partial_{\mathscr{A}_2}[cy_2+\mu y_1+\xi y_3+s(y_1^2+y_2y_3+y_3y_2)^{k-1}y_2+\lambda(y_1^2+y_2y_3+y_3y_2)^{k-1}y_1]\\
&\quad +t(y_1^2+y_2y_3+y_3y_2)^{k}
\end{align*}
which implies that $t=0$, and
\begin{align*}
b=&cy_2+\mu y_1+\xi y_3+s(y_1^2+y_2y_3+y_3y_2)^{k-1}y_2+\lambda(y_1^2+y_2y_3+y_3y_2)^{k-1}y_1\\
&\quad +\theta(y_1^2+y_2y_3+y_3y_2)^{k-1}y_3+\partial_{\mathscr{A}_2}(\beta)
\end{align*}
for some $\theta\in \k$ and $\beta\in \mathscr{A}_2^{2k-2}$.
Hence
\begin{align*}
&b_r\Sigma e_r + b_z\Sigma e_z + b_{y_3}\Sigma e_{y_3}+b\\
=&\partial_{\mathscr{A}_2}(c)\Sigma e_{r} +[\partial_{\mathscr{A}_2}(\mu)+cy_3+s(y_1^2+y_2y_3+y_3y_2)^{k-1}y_3] \Sigma e_{z} \\
&+[cy_1+\mu y_3+(y_1^2+y_2y_3+y_3y_2)^{k-1}(sy_1+\lambda y_3)+\partial_{\mathscr{A}_2}(\xi)]\Sigma e_{y_3}  \\
&+ cy_2+\mu y_1+\xi y_3+(y_1^2+y_2y_3+y_3y_2)^{k-1}(sy_2+\lambda y_1+\theta y_3)+\partial_{\mathscr{A}_2}(\beta)                     \\
 =&\partial_F\{c\Sigma e_r+\mu\Sigma e_z+\xi\Sigma e_{y_3}+\beta +(y_1^2+y_2y_3+y_3y_2)^{k-1}(s\Sigma e_r+\lambda \Sigma e_z +\theta \Sigma e_{y_3}) \}
\end{align*}
Hence $H^{2k-1}(F)=0$, for any $k\in \Bbb{N}$. Therefore, $f$ is a quasi-isomorphism.

Since $F$ has a semi-basis $\{1, \Sigma e_{y_3},\Sigma e_z, \Sigma e_r\}$ concentrated in degree $0$,  $\mathscr{A}_2$ is a Koszul homologically smooth DG algebra.
By the minimality of $F$, we have $$H(\Hom_{\mathscr{A}_2}(F,\k))=\Hom_{\mathscr{A}_2}(F,\k)= \k\cdot 1^*\oplus \k\cdot(\Sigma e_{y_3})^*\oplus \k \cdot(\Sigma e_z)^*\oplus \k \cdot (\Sigma e_r)^*.$$  So the Ext-algebra $E=H(\Hom_{\mathscr{A}_2}(F,F))$  is concentrated in degree $0$.
On the other hand, $$\Hom_{\mathscr{A}_2}(F,F)^{\#}\cong [\k \cdot 1^*\oplus \k \cdot (\Sigma e_{y_3})^*\oplus \k \cdot (\Sigma e_z)^*\oplus \k \cdot (\Sigma e_r)^*]\otimes_{\k} F^{\#}$$ is concentrated in degree $\ge 0$. This implies that $E= Z^0(\Hom_{\mathscr{A}_2}(F,F))$.
Since $F^{\#}$ is a free graded $\mathscr{A}_2^{\#}$-module with a basis $\{1,\Sigma e_{y_3},\Sigma e_z, \Sigma e_r\}$ concentrated in degree $0$,
  the elements in  $\Hom_{\mathscr{A}_2}(F,F)^0$ are in one to one correspondence with the matrices in $M_4(\k)$. Indeed, any $f\in \Hom_{\mathscr{A}_2}(F_{\k},F_{\k})^0$ is uniquely determined by
  a matrix $A_f=(a_{ij})_{4\times 4}\in M_4(\k)$ with
$$\left(
                         \begin{array}{c}
                          f(1) \\
                          f(\Sigma e_{y_3})\\
                          f(\Sigma e_z)\\
                          f(\Sigma e_r)
                         \end{array}
                       \right) =      A_f \cdot \left(
                         \begin{array}{c}
                          1 \\
                          \Sigma e_{x_3}\\
                          \Sigma e_z\\
                          \Sigma e_r
                         \end{array}
                       \right).  $$
                       We see $f\in  Z^0(\Hom_{\mathscr{A}_2}(F,F)$ if and only if $\partial_{F}\circ f=f\circ \partial_{F}$, if and only if
 $$ A_f\cdot \left(
                         \begin{array}{cccc}
                           0   & 0 & 0 &0\\
                           y_3 & 0 & 0&0 \\
                           y_1 & y_3 & 0&0 \\
                           y_2 & y_1 &y_3 &0
                         \end{array}
                       \right) = \left(
                         \begin{array}{cccc}
                           0   & 0 & 0 &0\\
                           y_3 & 0 & 0&0 \\
                           y_1 & y_3 & 0&0 \\
                           y_2 & y_1 &y_3 &0
                         \end{array}
                       \right) \cdot A_f, $$  which is also equivalent to
                       $$\begin{cases}
                       a_{12}=a_{13}=a_{14}=a_{23}=a_{24}=a_{34}=0\\
                       a_{11}=a_{22}=a_{33}=a_{44}\\
                       a_{21}=a_{32}=a_{43}\\
                        a_{42}=a_{31}
                       \end{cases}$$
by a direct computation. Let $$  \mathcal{E}=\left\{ \left(
                         \begin{array}{cccc}
                           a & 0 & 0 & 0\\
                           b & a & 0 & 0 \\
                           c & b & a & 0 \\
                           d & c & b & a
                         \end{array}
                       \right)\quad | \quad a,b,c, d\in \k \right\}$$ be the subalgebra of the matrix algebra. Then one sees $E\cong \mathcal{E}$.
                       Set \begin{align*} e_1= \left(
                         \begin{array}{cccc}
                           1 & 0 & 0 & 0\\
                           0 & 1 & 0 & 0 \\
                           0 & 0 & 1 & 0\\
                           0 & 0 & 0 & 1
                         \end{array}
                       \right), \quad & e_2= \left(
                         \begin{array}{cccc}
                           0 & 0& 0 & 0\\
                           1 & 0 & 0 & 0 \\
                           0 & 1 & 0 & 0 \\
                           0 & 0 & 1 & 0
                         \end{array}
                       \right), \\
                        e_3= \left(
                         \begin{array}{cccc}
                           0 & 0& 0 & 0\\
                           0 & 0 & 0 & 0 \\
                           1 & 0 & 0 & 0 \\
                           0 & 1 & 0 & 0
                         \end{array}
                       \right), \quad & e_4= \left(
                         \begin{array}{cccc}
                           0 & 0& 0 & 0\\
                           0 & 0 & 0 & 0 \\
                           0 & 0 & 0 & 0 \\
                           1 & 0 & 0 & 0
                         \end{array}
                       \right).
                       \end{align*}
                     Then $\{e_1,e_2,e_3, e_4\}$ is a $\k$-linear bases of the $\k$-algebra
                        $\mathcal{E}$. The multiplication on $\mathcal{E}$ is defined by the following relations
                       $$\begin{cases} e_1\cdot e_i=e_i\cdot e_1=e_i, i=1,2,3,4 \\
                        e_2^2=e_3, e_2\cdot e_3=e_3\cdot e_2=e_4,\\
                        e_2e_4=e_4e_2=0, e_3^2=0=e_4^2, e_3e_4=e_4e_3=0
                       \end{cases} .$$
                       So $\mathcal{E}$ is a local commutative $\k$-algebra isomorphic to $\k[X]/(X^4)$.  Hence
$E\cong \k[X]/(X^4)$ is a symmetric Frobenius algebra concentrated
in degree $0$. This implies that
$\mathrm{Tor}_{\mathscr{A}_2}^0(\k_{\mathscr{A}_2},{}_{\mathscr{A}_2}\k)\cong
E^*$ is a symmetric coalgebra. By \cite[Theorem 4.2]{HM},
$\mathscr{A}_2$ is a Koszul Calabi-Yau DG algebra. Since $\{e_1,e_2,e_3,e_4\}$ is a $\k$-linear basis of $\mathcal{E}$, any $\k$-linear map $\sigma: \mathcal{E}\to \mathcal{E}$ uniquely corresponds to a matrix in $A_{\sigma}=(a_{ij})_{4\times 4}\in M_4(\k)$ with
$$ \left(
     \begin{array}{c}
       \sigma(e_1) \\
       \sigma(e_2) \\
       \sigma(e_3) \\
       \sigma(e_4)
     \end{array}
   \right)=A_{\sigma}\left(
                       \begin{array}{c}
                         e_1 \\
                         e_2 \\
                         e_3 \\
                         e_4
                       \end{array}
                     \right).$$ Such $\sigma\in \mathrm{Aut}_{\k}(\mathcal{E})$ if and only if
                     $$A_{\sigma}\in \mathrm{GL}_4(\k) \quad \text{and}\quad \sigma(e_i\cdot e_j)=\sigma(e_i)\sigma(e_j), \, \text{for any}\, \, i,j =1,2,3,4.$$
Therefore, $\sigma\in \mathrm{Aut}_{\k}(\mathcal{E})$ if and only if
$$ \begin{cases}|(a_{ij})_{4\times 4}|\neq 0, \sigma(e_1)=e_1, [\sigma(e_3)]^2=0=[\sigma(e_4)]^2,\\
[\sigma(e_2)]^2=\sigma(e_3), \sigma(e_2\cdot e_3)=\sigma(e_3\cdot e_2)=\sigma(e_4)\\
\sigma(e_3)\cdot \sigma(e_4)=\sigma(e_4)\cdot \sigma(e_3)=0, \sigma(e_2)\cdot \sigma(e_4)=\sigma(e_4)\cdot \sigma(e_2)=0,
\end{cases}$$
if and only if $$\begin{cases} c_{22}\neq 0, c_{11}=1,c_{12}=c_{13}=c_{14}=0\\
c_{21}=c_{31}=c_{32}=c_{41}=c_{42}=c_{43}=0\\
c_{33}=c_{22}^2, c_{44}=c_{22}^3, c_{34}=2c_{22}c_{23}
\end{cases}$$
Hence \begin{align*}\mathrm{Aut}_{\k}\mathcal{E} &\cong \left\{\left(
                                            \begin{array}{cccc}
                                              1 & 0 & 0  & 0\\
                                              0 & a & b  & c \\
                                              0 & 0 & a^2&  2ab\\
                                              0 & 0 & 0  & a^3
                                            \end{array}
                                          \right)\,\,| \,\, a\in \k^{\times},b, c\in \k\right\}\\
                                          &\cong \left\{\left(
                                            \begin{array}{ccc}
                                               a & b  & c \\
                                               0 & a^2&  2ab\\
                                               0 & 0  & a^3
                                            \end{array}
                                          \right)\,\,| \,\, a\in \k^{\times},b, c\in \k\right\}
                                          \end{align*}
Since $\mathcal{E}$ is commutative, we have $\mathrm{Aut}_{\k}(\mathcal{E})\cong \mathrm{Out}_{\k}(\mathcal{E})$. By \cite[Remark 4.4]{MYH}, we have
$\mathrm{Pic}_{\k}(\mathcal{E})\cong \mathrm{Out}_{\k}(\mathcal{E})$ and $\mathrm{DPic}_{\k}(\mathcal{E})=\mathbb{Z}\times \mathrm{Pic}_{\k}(\mathcal{E})$. Thus $$\mathrm{DPic}(\mathscr{A}_2)\cong \mathrm{DPic}_{\k}(E)\cong \mathrm{DPic}_{\k}(\mathcal{E})\cong \mathbb{Z}\times \left\{\left(
                                            \begin{array}{ccc}
                                              a & b & c \\
                                              0 & a^2 & 2ab \\
                                              0 & 0 & a^3 \\
                                            \end{array}
                                          \right)\,\,| \,\, a\in \k^{\times},b,c\in \k\right\}$$
by \cite[Theorem 4.3]{MYH}.

\end{proof}

\section{explanations for the problem}
In this section, we will show that $\mathscr{A}_1$ and $\mathscr{A}_2$ can serve as examples to give a negative answer to Dugas's problem, which is also mentioned and studied in \cite{PPZ}.
Due to the computation results in Section \ref{Aone} and Section \ref{Atwo}, we can prove further that $\mathrm{DPic}(\mathscr{A}_1)\not \cong\mathrm{DPic}(\mathscr{A}_2)$.
\begin{prop}\label{notiso}
The groups $\mathrm{DPic}(\mathscr{A}_1)$ and $\mathrm{DPic}(\mathscr{A}_2)$ are not isomorphic.
\end{prop}

\begin{proof}
By Proposition \ref{aone} and Proposition \ref{atwo}, we have
$$\mathrm{DPic}(\mathscr{A}_1)\cong \mathbb{Z}\times  \left\{\left(
                                            \begin{array}{cc}
                                               a & b \\
                                               0 & a^2 \\
                                            \end{array}
                                          \right)\quad| \quad a\in \k^{\times},b\in \k\right\}$$
 and
$$\mathrm{DPic}(\mathscr{A}_2)\cong \mathbb{Z}\times \left\{\left(
                                            \begin{array}{ccc}
                                              a & b & c \\
                                              0 & a^2 & 2ab \\
                                              0 & 0 & a^3 \\
                                            \end{array}
                                          \right)\,\,| \,\, a\in \k^{\times},b,c\in \k\right\}.$$
Let $$G_1=\left\{\left(
                                            \begin{array}{cc}
                                               a & b \\
                                               0 & a^2 \\
                                            \end{array}
                                          \right)\quad| \quad a\in \k^{\times},b\in \k\right\}$$ and $$G_2=\left\{\left(
                                            \begin{array}{ccc}
                                              a & b & c \\
                                              0 & a^2 & 2ab \\
                                              0 & 0 & a^3 \\
                                            \end{array}
                                          \right)\,\,| \,\, a\in \k^{\times},b,c\in \k\right\}.$$
It suffices to show that the groups $G_1$ and $G_2$ are not isomorphic.
Define group homomorphisms
\begin{align*}
\sigma_1: &\quad\quad G_1 \quad \rightarrow \quad \k^{\times} \\
    &\left(
                                            \begin{array}{cc}
                                               a & b \\
                                               0 & a^2 \\
                                            \end{array}
                                          \right)   \mapsto \quad  a
\end{align*}
and
\begin{align*}
\sigma_2: &\quad\quad G_2 \quad \rightarrow \quad \k^{\times}\\
&\left(
                                            \begin{array}{ccc}
                                              a & b & c \\
                                              0 & a^2 & 2ab \\
                                              0 & 0 & a^3 \\
                                            \end{array}
                                          \right) \mapsto \quad  a.
\end{align*}
We have $$\mathrm{ker}(\sigma_1)=\left\{\left(
                                            \begin{array}{cc}
                                               1 & b \\
                                               0 & 1 \\
                                            \end{array}
                                          \right)\quad| \quad b\in \k\right\}\vartriangleleft G_1$$ and $$\mathrm{ker}(\sigma_2)=\left\{\left(
                                            \begin{array}{ccc}
                                              1 & b & c \\
                                              0 & 1 & 2b \\
                                              0 & 0 & 1 \\
                                            \end{array}
                                          \right)\,\,| \,\, b,c\in \k\right\} \vartriangleleft G_2.$$
Since $k^{\times}$ is abelian, one see easily that $[G_1,G_1]\subseteq \mathrm{ker}(\sigma_1)$ and $[G_2,G_2]\subseteq \mathrm{ker}(\sigma_1)$. On the other hand,  $$\forall \left(
                                            \begin{array}{cc}
                                               1 & b \\
                                               0 & 1 \\
                                            \end{array}
                                          \right)\in \mathrm{ker}(\sigma_1)\quad \text{and} \quad\left(
                                            \begin{array}{ccc}
                                              1 & b & c \\
                                              0 & 1 & 2b \\
                                              0 & 0 & 1 \\
                                            \end{array}
                                          \right)\in \mathrm{ker}(\sigma_2),$$ they can be respectively expressed by
$$\left(
                                            \begin{array}{cc}
                                               \frac{1}{2} & 0 \\
                                               0 & \frac{1}{4} \\
                                            \end{array}
                                          \right)\left(
                                            \begin{array}{cc}
                                               1 & b \\
                                               0 & 1 \\
                                            \end{array}
                                          \right)\left(
                                            \begin{array}{cc}
                                               \frac{1}{2} & 0 \\
                                               0 & \frac{1}{4} \\
                                            \end{array}
                                          \right)^{-1}\left(
                                            \begin{array}{cc}
                                               1 & b \\
                                               0 & 1 \\
                                            \end{array}
                                         \right)^{-1}$$
and $$
\left(
  \begin{array}{ccc}
    \frac{1}{2} & 0 & 0 \\
    0 & \frac{1}{4} & 0 \\
    0 & 0 & \frac{1}{8} \\
  \end{array}
\right)\left(
  \begin{array}{ccc}
    1 & b & \frac{c+2b^2}{3} \\
    0 & 1 & 2b \\
    0 & 0 & 1 \\
  \end{array}
\right)\left(
  \begin{array}{ccc}
    \frac{1}{2} & 0 & 0 \\
    0 & \frac{1}{4} & 0 \\
    0 & 0 & \frac{1}{8} \\
  \end{array}
\right)^{-1}\left(
  \begin{array}{ccc}
    1 & b & \frac{c+2b^2}{3} \\
    0 & 1 & 2b \\
    0 & 0 & 1 \\
  \end{array}
\right)^{-1}.
$$
Hence $$[G_1,G_1]=\mathrm{ker}(\sigma_1)=\left\{\left(
                                            \begin{array}{cc}
                                               1 & b \\
                                               0 & 1 \\
                                            \end{array}
                                          \right)\quad| \quad b\in \k\right\}$$ and $$[G_2,G_2]=\mathrm{ker}(\sigma_2)=\left\{\left(
                                            \begin{array}{ccc}
                                              1 & b & c \\
                                              0 & 1 & 2b \\
                                              0 & 0 & 1 \\
                                            \end{array}
                                         \right)\,\,| \,\, b,c\in \k\right\}.$$
Then $G_i/[G_i,G_i]\cong \k^{\times}$, $i=1,2$. One sees that $[G_1,G_1]$ and $[G_2,G_2]$ are abelian.
For any $i\in \{1,2\}$,  the group action of the quotient group $G_i/[G_i,G_i]$ on $[G_i,G_i]$ by conjugation is well defined. More precisely, we have the group action
\begin{align*}
\lambda_i:  G_i/[G_i,G_i]  &  \times [G_i,G_i] \to [G_i,G_i].\\
                 &(\overline{g_i}, c_i)\mapsto g_ic_ig_i^{-1}
\end{align*}
For any $$\left(
           \begin{array}{cc}
             a & b \\
             0 & a^2 \\
           \end{array}
         \right)\in G_1, \left(
                           \begin{array}{cc}
                             1 & s \\
                             0 & 1 \\
                           \end{array}
                         \right)\in [G_1,G_1]$$
 and $$\left(
           \begin{array}{ccc}
             a & b & c \\
             0 & a^2 & 2ab \\
             0 & 0 & a^3 \\
           \end{array}
         \right)\in G_2,
\left(
           \begin{array}{ccc}
             1 & s & t \\
             0 & 1 & 2s \\
             0 & 0 & 1 \\
           \end{array}
         \right)\in [G_2,G_2],$$ we have
$$\left(
           \begin{array}{cc}
             a & b \\
             0 & a^2 \\
           \end{array}
         \right)\left(
                           \begin{array}{cc}
                             1 & s \\
                             0 & 1 \\
                           \end{array}
                         \right)\left(
           \begin{array}{cc}
             a & b \\
             0 & a^2 \\
           \end{array}
         \right)^{-1}=\left(
                        \begin{array}{cc}
                          1 & \frac{s}{a} \\
                          0 & 1 \\
                        \end{array}
                      \right)$$
and
$$ \left(
           \begin{array}{ccc}
             a & b & c \\
             0 & a^2 & 2ab \\
             0 & 0 & a^3 \\
           \end{array}
         \right)\left(
           \begin{array}{ccc}
             1 & s & t \\
             0 & 1 & 2s \\
             0 & 0 & 1 \\
           \end{array}
         \right)\left(
           \begin{array}{ccc}
             a & b & c \\
             0 & a^2 & 2ab \\
             0 & 0 & a^3 \\
           \end{array}
         \right)^{-1}=\left(
                        \begin{array}{ccc}
                          1 & \frac{s}{a} & \frac{t}{a^2} \\
                          0 & 1 & \frac{2s}{a} \\
                          0 & 0 & 1 \\
                        \end{array}
                      \right).$$
These imply group actions of $\k^{\times}$ on $[G_1,G_1]$ and $[G_2,G_2]$ respectively defined by
\begin{align*}
\chi_1:\quad &\k^{\times} \times [G_1,G_1] \to [G_1,G_1] \\
&(a, \left(
       \begin{array}{cc}
         1 & s \\
         0 & 1 \\
       \end{array}
     \right))\mapsto \left(
                       \begin{array}{cc}
                         1 & \frac{s}{a} \\
                         0 & 1 \\
                       \end{array}
                     \right)
\end{align*}
and
\begin{align*}
\chi_2:\quad &\k^{\times} \times [G_2,G_2] \to [G_2,G_2] \\
&(a, \left(
       \begin{array}{ccc}
         1 & s & t \\
         0 & 1 & 2s \\
         0 & 0 & 1 \\
       \end{array}
     \right))\mapsto \left(
                        \begin{array}{ccc}
                          1 & \frac{s}{a} & \frac{t}{a^2} \\
                          0 & 1 & \frac{2s}{a} \\
                          0 & 0 & 1 \\
                        \end{array}
                      \right).
\end{align*}
It is easy to see that the $\k^{\times}$ invariant subgroups of $[G_1,G_1]$ under the action of $\chi_1$ are exactily the trivial subgroups $$\left(
                       \begin{array}{cc}
                         1 & 0 \\
                         0 & 1 \\
                       \end{array}
                     \right)\quad \text{and}\quad [G_1,G_1].$$
While the set of $\k^{\times}$ invariant subgroups of $[G_2,G_2]$ under the action of $\chi_2$ contains at least the following four distinct subgroups:
\begin{enumerate}
\item $\left(
       \begin{array}{ccc}
         1 & 0 & 0 \\
         0 & 1 & 0 \\
         0 & 0 & 1 \\
       \end{array}
     \right)$;
\item $\left\{\left(
       \begin{array}{ccc}
         1 & 0 & t \\
         0 & 1 & 0 \\
         0 & 0 & 1 \\
       \end{array}
     \right)\,\, |\,\, t\in \k\right\}$;
\item $\left\{\left(
       \begin{array}{ccc}
         1 & t & t^2 \\
         0 & 1 & 2t \\
         0 & 0 & 1 \\
       \end{array}
     \right)\,\, |\,\, t\in \k\right\}$;
\item $[G_2,G_2]$.
\end{enumerate}
If $G_1\cong G_2$, then there is a group isomorphism $\omega: G_1\to G_2$. It induces two group isomorphisms $\psi: [G_1,G_1]\to [G_2,G_2]$
and $\varepsilon: G_1/[G_1,G_1]\to G_2/[G_2,G_2]$, where $\psi=\omega|_{[G_1,G_1]}$ and $\varepsilon(\overline{g_1})=\overline{\omega(g_1)}$ for any $g_1\in G_1$. We have the following commutative diagram
\begin{align*}
\xymatrix{  &G_1/[G_1,G_1]\times [G_1,G_1] \ar[r]^{\quad\quad\quad\lambda_1}\ar[d]_{\varepsilon\times \psi} &  [G_1,G_1]\ar[d]^{\psi}
\\
  & G_2/[G_2,G_2]\times [G_2,G_2]  \ar[r]^{\quad\quad\quad\lambda_2} & [G_2,G_2],
\\}
\end{align*}
since $$\psi \circ \lambda_1(\overline{g_1}, c_1)=\psi(g_1c_1g_1^{-1})=\omega(g_1)\omega(c_1)\omega(g_1)^{-1}$$
and $$\lambda_2\circ (\varepsilon\times \psi)(\overline{g_1}, c_1)=\lambda_2(\overline{\omega(g_1)}, \omega(c_1))=\omega(g_1)\omega(c_1)\omega(g_1)^{-1}$$
for any $\overline{g_1}\in G_1/[G_1,G_1]$ and $c_1\in [G_1,G_1]$. This implies
the following commutative diagram
\begin{align*}
\xymatrix{  &\k^{\times}\times [G_1,G_1] \ar[r]^{\quad \chi_1}\ar[d]_{\overline{\sigma_2}\circ \varepsilon \circ \overline{\sigma_1}^{-1}\times \psi} &  [G_1,G_1]\ar[d]^{\psi}
\\
  & \k^{\times}\times [G_2,G_2]  \ar[r]^{\quad\chi_2} & [G_2,G_2].
\\}
\end{align*}
Hence the set of the invariant subgroups of $[G_1,G_1]$ under the group action $\chi_1$ is in one to one correspondence with the set of the invariants subgroups of $[G_2,G_2]$ under $\chi_2$. We reach a contradiction. So $G_1\not \cong G_2$ and $\mathrm{DPic}(\mathscr{A}_1)\not\cong \mathrm{DPic}(\mathscr{A}_2)$.
\end{proof}

 Since derived Picard group of an algebra is a derived invariant, we conclude that $\mathscr{A}_1$ and $\mathscr{A}_2$ are not derived equivalent by Proposition \ref{notiso}. On the other hand, we have $$H(\mathscr{A}_1)=\frac{\k[\lceil x_3 \rceil, \lceil x_1x_3+x_3x_1 \rceil]}{( \lceil x_3\rceil^2 )}\,\, \text{and}\,\, H(\mathscr{A}_2)=\frac{\k[\lceil y_3\rceil,\lceil y_1^2+y_2y_3+y_3y_2\rceil]}{(\lceil y_3\rceil^2 )}$$
by Proposition \ref{cohaone} and Proposition \ref{cohatwo}, respectively. One sees that $H(\mathscr{A}_1)
\cong H(\mathscr{A}_2)$ as graded algebras. It indicates that the derived equivalence of DG algebras are not governed by their cohomology rings.
We give a negative answer to Dugas's problem in \cite[Section 7]{Dug}. The reason for this phenomenon lies in the fact that much intrinsic information of a non-formal DG algebra $\mathscr{A}$ is missing after taking the cohomology functor. To retrieve the lost information, we should endow the cohomology ring with suitable $A_{\infty}$-algebra structures. Indeed, it is well known that $H(\mathscr{A})$ admits an $A_{\infty}$-algebra structure called the minimal model of $\mathscr{A}$ such that there is an $\mathscr{A}_{\infty}$-quasi-isomorphism $\mathscr{A}\to H(\mathscr{A})$. Obviously, cochain DG algebras $\mathscr{A}$ and $\mathscr{B}$ are derived equivalent, if their minimal models are derived equivalent.

\section*{Acknowledgments}
X.-F. Mao is supported by the National Natural Science Foundation of
China (No. 11871326).

\def\refname{References}

\end{document}